\documentclass{article}

\title{\huge\textbf{New Developments on the Non-Central Chi-Squared and Beta Distributions}\\
\Large{accepted for publication at \textit{Austrian Journal of Statistics} on 06 May 2021}}
\author{
Carlo Orsi\footnote{e-mail: \mail{orsi.carlo@gmail.com}}
}
\usepackage[utf8]{inputenc}
\usepackage[english]{babel}
\usepackage{geometry}
\usepackage{fancyhdr}
\usepackage{hyperref}
\hypersetup{colorlinks=true, linkcolor=black, citecolor=black, urlcolor=black, plainpages=false}
\usepackage{color}
\usepackage{nameref}
\usepackage{graphicx}
\usepackage{subfigure}
\usepackage{float}
\usepackage{amsthm}
\usepackage{amssymb}
\usepackage{amsmath}
\usepackage{verbatim}
\usepackage{lscape}
\usepackage[font=small,labelfont=bf]{caption}
\usepackage{multirow}
\usepackage{wrapfig}
\newcommand{\mail}[1]{\href{mailto:#1}{\texttt{#1}}}

\usepackage[table]{xcolor}
\usepackage{array}
\newcolumntype{C}[1]{>{\centering\let\newline\\\arraybackslash\hspace{0pt}}m{#1}}
\newcolumntype{L}[1]{>{\raggedright\let\newline\\\arraybackslash\hspace{0pt}}m{#1}}
\newcolumntype{R}[1]{>{\raggedleft\let\newline\\\arraybackslash\hspace{0pt}}m{#1}}
\newtheoremstyle{plain}  
  {\topsep}   
  {\topsep}   
  {\itshape}  
  {}       
  {\bfseries} 
  {}         
  {\newline}  
  {}          
\theoremstyle{plain}
\newtheorem{property}{Property}[section]
\newtheorem{proposition}{\textbf{Proposition}}[section]

\newtheoremstyle{mystyle}  
  {\topsep}   
  {\topsep}   
  {\normalfont}  
  {0pt}       
  {\bfseries} 
  {.}         
  {\newline}  
  {}          
\theoremstyle{mystyle}

\begin{document}
\baselineskip24pt
\maketitle

\begin{abstract}
\baselineskip24pt
New formulas for the moments about zero of the Non-central Chi-Squared and the Non-central Beta distributions are achieved by means of novel approaches. The mixture representation of the former model and a new expansion of the ascending factorial of a binomial are the main ingredients of the first approach, whereas the second one hinges on an interesting relationship of conditional independence and a simple conditional density of the latter model. Then, a simulation study is carried out in order to pursue a twofold purpose: providing numerical validations of the derived moment formulas on one side and discussing the advantages of the new formulas over the existing ones on the other.

\vspace*{0.2cm}
\end{abstract}

\section{Introduction}
\label{sec:introduc}

In the present paper new expressions for the $r$-th moment about zero of the Non-central Chi-Squared and the Non-central Beta distributions are obtained. In order to go in due depth into the topics of interest, Section~\ref{sec:prelim} is devoted to recalling the definitions and the main properties of the above mentioned distributions. Instead, the core of the matter is addressed in the subsequent sections as follows. In Section~\ref{sec:form.mom} the problem of computing the $r$-th moment about zero of the Non-central Chi-Squared distribution is faced by resorting to a novel approach based on the mixture representation of such a distribution together with a new expansion of the ascending factorial of a binomial. In Section~\ref{sec:dnc.beta.mom} an approach to the analysis of the Doubly Non-central Beta distribution, i.e. the most general non-central extension of the Beta one, is made explicit. This approach rests on an interesting relationship of conditional independence and a suitable conditional density of such a model. These findings provide an analytical tool-kit that paves the way towards obtaining a surprisingly simple solution to the issue of assessing the moments of such a distribution. Finally, Section~\ref{sec:simul_res} contains a simulation study aimed at discussing the validity of the derived moment formulas and the advantages of these latter formulas over the existing ones. The whole analysis is performed by using the statistical software \texttt{R}. 

\section{Preliminaries}
\label{sec:prelim}

This section is intended to firstly provide the main mathematical tools that will be used in the sequel. More precisely, let $a > 0$ and $\Gamma\left(\cdot\right)$ be the gamma function; we remember that
\begin{equation}
\left(a\right)_l=\frac{\Gamma\left(a+l\right)}{\Gamma\left(a\right)}=\left\{\begin{array}{ll} 1 & \mbox{ if } l=0 \\ a\left(a+1\right)\ldots\left(a+l-1\right) & \mbox{ if } l \in \mathbb{N}\end{array} \right. 
\label{eq:poch.symb}
\end{equation}
denotes the ascending factorial or Pochhammer's symbol of $a$ \cite{JohKemKot05}. By Equation~(\ref{eq:poch.symb}), one has
\begin{equation}
\left(a\right)_{l+m}=\frac{\Gamma\left(a+l+m\right)}{\Gamma\left(a\right)}=\left\{\begin{array}{l} \frac{\Gamma\left(a+l\right)}{\Gamma\left(a\right)} \, \frac{\Gamma\left(a+l+m\right)}{\Gamma\left(a+l\right)}=\left(a\right)_l \, \left(a+l\right)_m \\ \\ \frac{\Gamma\left(a+m\right)}{\Gamma\left(a\right)} \, \frac{\Gamma\left(a+m+l\right)}{\Gamma\left(a+m\right)}=\left(a\right)_m \, \left(a+m\right)_l\end{array}\right.
\label{eq:poch.symb.sum}
\end{equation}
for every $l,m \in \mathbb{N}\cup \{0\}$ and
\begin{equation}
\frac{\left(a\right)_l}{\left(a\right)_m}=\left\{\begin{array}{ll} \left(a+m\right)_{l-m} & \mbox{ if } l>m \\ 1/\left(a+l\right)_{m-l} & \mbox{ if } l<m\end{array} \right. \, . 
\label{eq:poch.symb.ratio}
\end{equation}
Furthermore, the following formula holds for the Pochhammer's symbol of a binomial:
\begin{equation}
\left(a+b\right)_{n}=\sum_{j=0}^{n} {n \choose j} \left(a\right)_{n-j} \left(b\right)_j \, .
\label{eq:poch.symb.binom}
\end{equation}
In view of the foregoing arguments, the generalized hypergeometric function with $p$ numerator parameters and $q$ denominator parameters can be defined as follows:
\begin{equation}
_p^{\, }F_q^{}\left(a_1,\ldots,a_p; b_1,\ldots,b_q; x\right)=\sum_{i=0}^{+\infty} \frac{(a_1)_i \ldots (a_p)_i}{(b_1)_i \ldots (b_q)_i} \frac{x^{i}}{i \, !} \, , \quad x \in \mathbb{R} \, .
\label{eq:fpq}
\end{equation}
For more details on the convergence of the hypergeometric series in Equation~(\ref{eq:fpq}) as well as for further results and properties of $_p^{\, }F_q^{} \, $, the reader is recommended referring to \cite{SriKar85}. An important role is played herein by the special case of Equation~(\ref{eq:fpq}) where both $p$ and $q$ are set equal to 1, namely
\begin{equation}
_1^{\, }F_1^{}\left(a; b; x\right)=\sum_{i=0}^{+\infty} \frac{(a)_i}{(b)_i} \frac{x^{i}}{i \, !} \, , \quad x \in \mathbb{R} \, ,
\label{eq:f11}
\end{equation}
which is known as Kummer's confluent hypergeometric function. Some well-established recurrence relations holding for contiguous values of the parameters of $_1^{\, }F_1^{}$ are of special interest for the subsequent analysis and therefore are recalled in the following:
\begin{equation}
\left(b-a\right) \, _1^{\, }F_1^{}\left(a-1; b; x\right)+\left(2a-b+x\right) \, _1^{\, }F_1^{}\left(a; b; x\right)-a \, _1^{\, }F_1^{}\left(a+1; b; x\right)=0 \, ,
\label{eq:f11.rec.rel1}
\end{equation}
\begin{equation}
b\left(b-1\right) \, _1^{\, }F_1^{}\left(a; b-1; x\right)+b\left(1-b-x\right) \, _1^{\, }F_1^{}\left(a; b; x\right)+x\left(b-a\right) \, _1^{\, }F_1^{}\left(a; b+1; x\right)=0 \, ,
\label{eq:f11.rec.rel2}
\end{equation}
\begin{equation}
b\left(a+x\right) \, _1^{\, }F_1^{}\left(a; b; x\right)+x\left(a-b\right) \, _1^{\, }F_1^{}\left(a; b+1; x\right)-ab \, _1^{\, }F_1^{}\left(a+1; b; x\right)=0 \, ,
\label{eq:f11.rec.rel3}
\end{equation}
\begin{equation}
\left(a-1+x\right) \, _1^{\, }F_1^{}\left(a; b; x\right)+\left(b-a\right) \, _1^{\, }F_1^{}\left(a-1; b; x\right)+\left(1-b\right) \, _1^{\, }F_1^{}\left(a; b-1; x\right)=0 \, 
\label{eq:f11.rec.rel4}
\end{equation}
(see \cite{AbrSte64}, formulas 13.4.1, 13.4.2, 13.4.5, 13.4.6).

That said, let us proceed with the reviewing of the probabilistic models which are under consideration in the present paper. To begin with, we recall the following characterizing property of independent Chi-Squared and, more generally, Gamma random variables \cite{JohKotBal94}. This property is a matter of great consequence for our interests. Thus, it is given the following special emphasis.

\begin{property}[Characterizing property of independent Chi-Squared random variables]
\label{prope:char.prop.chisq}
Let $Y_i$, $i=1,2$, be independent Chi-Squared random variables and $Y^+=Y_1+Y_2$. Then, each of the compositional ratios $Y_i \, / \, Y^+$ is independent of $Y^+$.
\end{property}

In the notation of Property~\ref{prope:char.prop.chisq}, let $2\alpha_i >0$ be the number of degrees of freedom of $Y_i$, $i=1,2$; then, the random variable $X=Y_1 \, / \, Y^+$ is said to have a Beta distribution with shape parameters $\alpha_1, \alpha_2$. We shall denote it by Beta$(\alpha_1, \alpha_2)$. Meaning that a $\chi^{\, 2}_0$ random variable is equal to zero with probability one, a $\mbox{Beta}\left(\alpha_1,0\right)$ random variable with $\alpha_1>0$ is interpreted as degenerate at one, while a $\mbox{Beta}\left(0,\alpha_2\right)$ random variable with $\alpha_2>0$ is interpreted as degenerate at zero. As a consequence of Property~\ref{prope:char.prop.chisq}, the Beta distribution can also be obtained as conditional distribution of $X$ given $Y^+$. The Beta density function takes the form of
$$\mbox{Beta}\left(x;\alpha_1,\alpha_2\right)=\frac{x^{\alpha_1-1} \, \left(1-x\right)^{\alpha_2-1}}{B\left(\alpha_1,\alpha_2\right)} \; , \quad 0<x<1 \, ,$$where $B\left(\alpha_1 \, , \alpha_2\right)=\Gamma\left(\alpha_1\right) \, \Gamma\left(\alpha_2\right)/ \, \Gamma\left(\alpha^+\right)$ is the beta function and $\alpha^+=\alpha_1+\alpha_2$. Finally, in light of Equation~(\ref{eq:poch.symb}), the $r$-th moment about zero of the $\mbox{Beta}\left(\alpha_1,\alpha_2\right)$ distribution can be stated in the form of
\begin{equation}
\mathbb{E}\left[\mbox{Beta}^{\, r} \left(\alpha_1,\alpha_2\right)\right]=\frac{\left(\alpha_1\right)_r}{\left(\alpha^+\right)_r} \; , \quad r \in \mathbb{N} \, .
\label{eq:momr.beta}
\end{equation}
Another probability distribution often mentioned in the following is the Poisson one. Some properties of the latter are listed below by making use of the same notation that will be adopted for our purposes afterwards. Specifically, a discrete random variable $M$ is said to have a Poisson distribution with parameter $\lambda/2 \geq 0$, denoted by $\mbox{Poisson}\left(\lambda/2\right)$, if its probability mass function can be expressed as
$$\mbox{Poisson}\left(j \, ; \, \frac{\lambda}{2}\right)=e^{-\frac{\lambda}{2}} \, \frac{\left(\frac{\lambda}{2}\right)^j}{j !} \; , \quad j \in \mathbb{N} \cup \{0\} \, .$$The case $\lambda=0$ corresponds to a random variable degenerate at zero. If $M \sim \mbox{Poisson}\left(\lambda/2\right)$, then $\mathbb{E}\left(M\right)=\mathbb{V}\left(M\right)=\lambda/2$. The Poisson family is reproductive with respect to its parameter. More precisely, let $\left(M_1,M_2\right)$ be a bivariate random variable the marginals of which are independent with $\mbox{Poisson}\left(\lambda_i \, / \, 2\right)$ distributions, $i=1,2$, respectively; then $M^+=M_1+M_2 \sim \mbox{Poisson}\left(\lambda^+/ \, 2\right)$ with $\lambda^+=\lambda_1+\lambda_2$. Furthermore, it is useful remembering that $\left.M_i \, \right| \, M^+ \sim \mbox{Binomial}\left(M^+,\lambda_i \, / \, \lambda^+\right)$, $i=1,2$. The Poisson distribution is instrumental in characterizing the non-central extension of the Chi-Squared model \cite{JohKotBal95}. As a matter of fact, a Non-central Chi-Squared random variable $Y'$ with $g>0$ degrees of freedom and non-centrality parameter $\lambda \ge 0$, denoted by $\chi'^{\,2}_g \left(\lambda \right)$, admits the following mixture representation:
\begin{equation}
Y' \sim \chi'^{\,2}_g \left(\lambda \right) \qquad \Leftrightarrow \qquad Y'\,| \, M \; \sim \; \chi^{\, 2}_{g+2M} \, , \qquad \mbox{where} \; \; M \sim \mbox{Poisson}\left(\frac{\lambda}{2}\right) \, ,
\label{eq:mixrepres.ncchisq}
\end{equation}
the case $\lambda=0$ corresponding to the $\chi^{\, 2}_g$ distribution. Moreover, a random variable $Y' \sim \chi'^{\,2}_g \left(\lambda \right)$ can be additively decomposed into a central part with $g$ degrees of freedom and a purely non-central part with non-centrality parameter $\lambda$, as follows:
\begin{equation}
Y'=Y+\sum_{j=1}^{M}F_j \, ,
\label{eq:sumrepres.ncchisq}
\end{equation}
where:
\begin{itemize}
\item[i)] $Y$, $M$, $\left\{F_j\right\}$ are mutually independent,
\item[ii)] $Y \sim \chi^{\, 2}_g$, $M \sim \mbox{\normalfont{Poisson}}\left(\lambda/2\right)$ and $\left\{F_j\right\}$ is a sequence of independent random variables with $\chi^{\, 2}_2$ distribution.  
\end{itemize}
By virtue of Equation~(\ref{eq:sumrepres.ncchisq}), the random variable $Y'_{pnc}=\sum_{j=1}^{M}F_j$ is said to have a Purely Non-central Chi-Squared distribution with non-centrality parameter $\lambda$. Indeed, we shall denote it by $\chi'^{\,2}_0 \left(\lambda \right)$, its number of degrees of freedom being equal to zero \cite{Sie79}. The Non-central Chi-Squared distribution is reproductive with respect to both the number of degrees of freedom and the non-centrality parameter \cite{JohKotBal95}. Specifically, if $Y'_1,\ldots,Y'_m$ are independent with $\chi'^{\, 2}_{g_j}(\lambda_j)$ distributions, $j=1,\ldots,m$, then $Y'^{+}=\sum_{j=1}^m Y'_j \sim \chi'^{\, 2}_{g^+}(\lambda^+)$ with $g^+=\sum_{j=1}^m g_j$ and $\lambda^+=\sum_{j=1}^m \lambda_j$.

Finally, the $r$-th moment about zero of $Y' \sim \chi'^{\,2}_g \left(\lambda \right)$, $g>0$, can be computed with the following formula \cite{JohKotBal95}:
\begin{equation}
\mathbb{E}\left(\chi'^{\,2}_g \left(\lambda \right) \right)^r=2^r \, \Gamma\left(r+\frac{g}{2}\right) \sum_{j=0}^{r} {r \choose j} \frac{\left(\frac{\lambda}{2}\right)^j}{\Gamma\left(j+\frac{g}{2}\right)} \, .
\label{eq:mom.literat.ncchisq}
\end{equation}

The Non-central Chi-Squared distribution represents the main ingredient for the definition of the most general non-central extension of the Beta model, known as Doubly Non-central Beta model \cite{JohKotBal95}. The latter is defined by replacing the two independent Chi-Squared random variables involved in the definition of the Beta model by two independent Non-central Chi-Squareds. More precisely, a random variable is said to have a Doubly Non-central Beta distribution with shape parameters $\alpha_1, \, \alpha_2$ and non-centrality parameters $\lambda_1, \, \lambda_2$, denoted by $\mbox{\normalfont{DNcB}}\left(\alpha_1,\alpha_2,\lambda_1,\lambda_2\right)$, if it is distributed as $X'=Y'_1/\left(Y'_1+Y'_2\right)$, the $Y'_i \,$'s being independent with $\chi'^{\,2}_{2\alpha_i}\left(\lambda_i\right)$ distributions, $i=1,2$. Clearly, the case $\lambda_1=\lambda_2=0$ corresponds to the Beta distribution. Moreover, the special case of $X'$ where $\alpha_1=\alpha_2=0$ leads to the compositional ratio $X'_{pnc}$ of two Purely Non-central Chi-Squared independent random variables with non-centrality parameters $\lambda_1$, $\lambda_2$. In light of Equation~(\ref{eq:mixrepres.ncchisq}), the Doubly Non-central Beta model can be characterized by the following mixture representation:
\begin{eqnarray}
X' \sim \mbox{\normalfont{DNcB}}\left(\alpha_1,\alpha_2,\lambda_1,\lambda_2\right) & \quad \Leftrightarrow \quad & X'\,| \, \left(M_1,M_2\right) \; \sim \; \mbox{Beta}(\alpha_1+M_1, \alpha_2 + M_2) \nonumber \\
&  & \mbox{where} \; \; M_i \stackrel{\mbox{\tiny ind}}{\sim } \mbox{Poisson}\left(\frac{\lambda_i}{2}\right)  \; \; i=1,2 \, .
\label{eq:mixrepres.dnc}
\end{eqnarray}
By Equation~(\ref{eq:mixrepres.dnc}), the $\mbox{\normalfont{DNcB}}$ density can be stated as
\begin{eqnarray}
\lefteqn{\mbox{\normalfont{DNcB}}\left(x;\alpha_1,\alpha_2,\lambda_1,\lambda_2\right)= \qquad \qquad \qquad \qquad \qquad \qquad 0<x<1}\nonumber \\
& = & \sum_{j=0}^{+\infty} \sum_{k=0}^{+\infty} \left[ \mbox{\normalfont{Poisson}}\left(j \, ; \, \frac{\lambda_1}{2}\right) \mbox{\normalfont{Poisson}}\left(k \, ; \, \frac{\lambda_2}{2}\right) \mbox{\normalfont{Beta}}\left(x; \alpha_1+j,\alpha_2+k\right)\right] 
\label{eq:dens.beta.dnc}
\end{eqnarray}
and can be accordingly written as the following perturbation of the Beta density:
\begin{eqnarray}
\lefteqn{\mbox{\normalfont{DNcB}}\left(x;\alpha_1,\alpha_2,\lambda_1,\lambda_2\right)= \qquad \qquad \qquad \qquad 0<x<1}\nonumber \\
& = & \mbox{\normalfont{Beta}}\left(x; \alpha_1,\alpha_2 \right) \cdot e^{-\frac{\lambda^+}{2}} \,  \Psi_2\left[\alpha^+;\alpha_1,\alpha_2;\frac{\lambda_1}{2}x,\frac{\lambda_2 }{2}\left(1-x\right)\right] \; ,
\label{eq:c.ncb}
\end{eqnarray}
\cite{OngOrs15} where
\begin{equation}
\Psi_2\left[a;b_1,b_2;x,y\right]=\sum_{j=0}^{+\infty}\sum_{k=0}^{+\infty}\frac{(a)_{j+k}}{(b_1)_j \, (b_2)_k}  \frac{x^j}{j!}\frac{y^k}{k!}, \quad x,y \geq 0
\label{eq:perturb.ncb}
\end{equation}
is the Humbert's confluent hypergeometric function \cite{SriKar85}. Unfortunately, the perturbation representation of the DNcB density in Equation~(\ref{eq:c.ncb}) is not so easily tractable and interpretable. Indeed, regardless of the constant term, the Beta density is perturbed by a function in two variables given by the sum of the double power series in Equation~(\ref{eq:perturb.ncb}).

The special case of the DNcB model where $\lambda_2$ is set equal to 0 and $\lambda_1$ is renamed in $\lambda$ is called Type I Non-central Beta distribution and is denoted by $X'_1 \sim \mbox{\normalfont{NcB1}}\left(\alpha_1,\alpha_2,\lambda\right)$. The NcB1 density takes the form of
$$\mbox{\normalfont{NcB1}}\left(x;\alpha_1,\alpha_2,\lambda\right)=\sum_{j=0}^{+\infty} \left[\mbox{\normalfont{Poisson}}\left(j \, ; \, \frac{\lambda}{2}\right) \mbox{\normalfont{Beta}}\left(x; \alpha_1+j,\alpha_2\right)\right]  \, ,\quad 0<x<1$$
and can be equivalently written as
$$\mbox{\normalfont{NcB1}}\left(x;\alpha_1,\alpha_2,\lambda\right)= \mbox{\normalfont{Beta}}\left(x; \alpha_1,\alpha_2 \right) \cdot e^{-\frac{\lambda}{2}} \, _1 F_1\left(\alpha^+;\alpha_1;\frac{\lambda}{2} \, x\right) \, ;$$
moreover, the following formula holds for the raw moments of a $\mbox{\normalfont{NcB1}}\left(\alpha_1,\alpha_2,\lambda\right)$ random variable:
\begin{equation}
\mathbb{E}\left[\mbox{\normalfont{NcB1}}^{\, r}\left(\alpha_1,\alpha_2,\lambda\right)\right]=\mathbb{E}\left[\mbox{Beta}^{\, r} \left(\alpha_1,\alpha_2\right)\right] \, e^{-\frac{\lambda}{2}} \, _2F_2\left(\alpha_1+r,\alpha^+;\alpha_1, \alpha^++r;\frac{\lambda}{2}\right) \, , \; \; r \in \mathbb{N} \, .
\label{eq:momr.beta.nc1.def}
\end{equation}
The Type II Non-central Beta distribution, denoted by $\mbox{\normalfont{NcB2}}\left(\alpha_1,\alpha_2,\lambda\right)$, is the special case of the $\mbox{\normalfont{DNcB}}$ model given by the distribution of $1-X'_1$ where $X'_1$ follows a $\mbox{\normalfont{NcB1}}$ distribution with reversed shape parameters. The $r-$th moment about zero of a $\mbox{\normalfont{NcB2}}\left(\alpha_1,\alpha_2,\lambda\right)$ random variable can be expressed as:
\begin{equation}
\mathbb{E}\left[\mbox{\normalfont{NcB2}}^{\, r}\left(\alpha_1,\alpha_2,\lambda\right)\right]=\mathbb{E}\left[\mbox{Beta}^{\, r} \left(\alpha_1,\alpha_2\right)\right]  \, e^{-\frac{\lambda}{2}} \, _1F_1\left(\alpha^+;\alpha^++r;\frac{\lambda}{2}\right) \, , \; \; r \in \mathbb{N} \, .
\label{eq:momr.beta.nc2}
\end{equation}
See \cite{NadGup04} for more details of these two latter models.

\section{On the moments of the Non-central Chi-Squared distribution}
\label{sec:form.mom}

In the present section the focus is on the derivation of a new general expression for the raw moments of the Non-central Chi-Squared distribution. This new formula is derived regardless of the standard one in Equation~(\ref{eq:mom.literat.ncchisq}) by means of a novel conditional approach. This approach makes use of the mixture representation of the above mentioned distribution as well as the following simple expansion of the ascending factorial of a binomial, which, to our knowledge, has never been discussed in the literature.

\begin{proposition}[Expansion of the ascending factorial of a binomial]
\label{prop:expans.poch.symb.binom}
Let $a, \, b > 0$. Then, for every $l \in \mathbb{N} \cup \{0\}$:
\begin{equation}
\left(a+b\right)_l=\sum_{i=0}^{l}\frac{1}{i!}\left[\frac{d^{\, i}}{d a^i} \left(a\right)_l\right] \, b^i,
\label{eq:expans.poch.symb.binom}
\end{equation}
where $d^{\, i} f/d a^i$ denotes the $i$-th derivative of $f$ with respect to $a$ (the case $i=0$ corresponding to $f$) and $\left(a\right)_l$ is defined as in Equation~(\ref{eq:poch.symb}).
\end{proposition}
\begin{proof}
In view of Equation~(\ref{eq:poch.symb}), Equation~(\ref{eq:expans.poch.symb.binom}) holds trivially true in the case where $l$ is set equal to 0. By Equation~(\ref{eq:poch.symb.sum}), for $l \ge 1$ one has
\begin{equation}
\left(a+b\right)_l=\left(a+b\right)_{1+\left(l-1\right)}=\left(a+b\right) \, \left(a+1+b\right)_{l-1} \, ,
\label{eq:expans.poch.symb.binom.dim1}
\end{equation}
which can be equivalently rewritten as
\begin{equation}
\left(a+b\right)_l=\left\{\sum_{i=0}^{1}\frac{1}{i!}\left[\frac{d^i}{d a^i} \left(a\right)_1\right] \, b^i \right\} \left(a+1+b\right)_{l-1} \, ;
\label{eq:expans.poch.symb.binom.dim1a}
\end{equation}
therefore, by taking $l=1$ in Equation~(\ref{eq:expans.poch.symb.binom.dim1a}), Equation~(\ref{eq:expans.poch.symb.binom}) is established in this latter case. For $l \ge 2$, Equation~(\ref{eq:expans.poch.symb.binom.dim1}) can be further expanded in the following way:
\begin{eqnarray}
\left(a+b\right)_l & = & \left(a+b\right) \, \left(a+1+b\right)_{1+\left(l-2\right)}=\left(a+b\right) \, \left(a+1+b\right) \, \left(a+2+b\right)_{l-2}=\nonumber\\
& = & \left\{a\left(a+1\right)+\left[a+\left(a+1\right)\right]b+b^2\right\} \, \left(a+2+b\right)_{l-2} \, .
\label{eq:expans.poch.symb.binom.dim2}
\end{eqnarray}
Observe that Equation~(\ref{eq:expans.poch.symb.binom.dim2}) can be analogously restated as
\begin{equation}
\left(a+b\right)_l=\left\{\sum_{i=0}^{2}\frac{1}{i!}\left[\frac{d^i}{d a^i} \left(a\right)_2\right] \, b^i \right\} \left(a+2+b\right)_{l-2} \, ,
\label{eq:expans.poch.symb.binom.dim2a}
\end{equation}
so that, by placing $l=2$ in Equation~(\ref{eq:expans.poch.symb.binom.dim2a}), Equation~(\ref{eq:expans.poch.symb.binom}) holds true also in this latter case. By expanding Equation~(\ref{eq:expans.poch.symb.binom.dim2}) along the same lines, for $l \ge 3$ one obtains:
\begin{eqnarray}
\left(a+b\right)_l & = & \left\{a\left(a+1\right)+\left[a+\left(a+1\right)\right]b+b^2\right\} \, \left(a+2+b\right)_{1+\left(l-3\right)}=\nonumber\\
& = & \left\{a\left(a+1\right)+\left[a+\left(a+1\right)\right]b+b^2\right\} \, \left(a+2+b\right) \, \left(a+3+b\right)_{l-3}=\nonumber\\
& = & \left\{a\left(a+1\right)\left(a+2\right)+\left[a\left(a+1\right)+a\left(a+2\right)+\left(a+1\right)\left(a+2\right)\right]b \, +\right.\nonumber\\
& + & \left.\left[a+\left(a+1\right)+\left(a+2\right)\right]b^2+b^3\right\} \, \left(a+3+b\right)_{l-3}.
\label{eq:expans.poch.symb.binom.dim3}
\end{eqnarray}
Again note that Equation~(\ref{eq:expans.poch.symb.binom.dim3}) can be rewritten as follows:
\begin{equation}
\left(a+b\right)_l=\left\{\sum_{i=0}^{3}\frac{1}{i!}\left[\frac{d^i}{d a^i} \left(a\right)_3\right] \, b^i \right\} \left(a+3+b\right)_{l-3} \, ;
\label{eq:expans.poch.symb.binom.dim3a}
\end{equation}
therefore, the special case of Equation~(\ref{eq:expans.poch.symb.binom.dim3a}) where $l$ is set equal to 3 leads to the proof of Equation~(\ref{eq:expans.poch.symb.binom}) also in this latter case. By iteratively repeating the foregoing reasonings, at the $j$-th step ($l \geq j$) one has
\begin{equation}
\left(a+b\right)_l=\left\{\sum_{i=0}^{j}\frac{1}{i!}\left[\frac{d^i}{d a^i} \left(a\right)_j\right] \, b^i \right\} \left(a+j+b\right)_{l-j} \, ;
\label{eq:expans.poch.symb.binom.dimj}
\end{equation}
therefore, by setting $j=l$ in Equation~(\ref{eq:expans.poch.symb.binom.dimj}), Equation~(\ref{eq:expans.poch.symb.binom}) is established for any $l \in \mathbb{N} \cup \{0\}$.
\end{proof}
As mentioned before, Proposition~\ref{prop:expans.poch.symb.binom} and the mixture representation in Equation~(\ref{eq:mixrepres.ncchisq}) lead us to the finding of the following new general formula for the moments about zero of the Non-central Chi-Squared distribution.

\begin{proposition}[Moments about zero of the $\chi'^{\,2}_g \left(\lambda \right)$ distribution]
\label{prop:mom.ncchisq}
For every $r \in \mathbb{N}$, the $r$-th moment about zero of the $\chi'^{\,2}_g \left(\lambda \right)$ distribution with $g>0$ can be written as
\begin{equation}
\mathbb{E}\left(\chi'^{\,2}_g \left(\lambda \right) \right)^r=2^{\, r} \, \sum_{i=0}^{r}\sum_{j=0}^{i} \mathcal{S}\left(i,j\right) \frac{1}{i!}\left[\frac{d^{\, i}}{d h^i} \left(h\right)_r\right] \left(\frac{\lambda}{2}\right)^j \, ,
\label{eq:mom.ncchisq}
\end{equation}
where $\mathcal{S}\left(i,j\right)$ is a Stirling number of the second kind, $h=g/2$, $d^{\, i} f/d h^i$ denotes the $i$-th derivative of $f$ with respect to $h$ (the case $i=0$ corresponding to $f$) and $\left(h\right)_r$ is defined as in Equation~(\ref{eq:poch.symb}).
\end{proposition}
\begin{proof}
In the notation of Equation~(\ref{eq:mixrepres.ncchisq}), by the law of iterated expectations, one has $\mathbb{E}\left[\left(Y' \right)^r \right]=\mathbb{E}_M\left\{\mathbb{E}\left[\left.\left(Y' \right)^r\right|M \right]\right\}$. In view of the general formula for the moments about zero of the Gamma distribution \cite{JohKotBal94}, one obtains $\mathbb{E}\left[\left.\left(Y' \right)^r\right|M \right]=2^{\, r} \left(h+M\right)_r \,$; therefore:
\begin{equation}
\mathbb{E}\left[\left(Y' \right)^r \right]=2^r \, \mathbb{E}\left[\left(h+M\right)_r\right].
\label{eq:mom.ncchisq.dim1}
\end{equation}
By virtue of the special case of Equation~(\ref{eq:expans.poch.symb.binom}) where $a=h$, $b=M$ and $l=r$, Equation~(\ref{eq:mom.ncchisq.dim1}) can be rewritten as
$$\mathbb{E}\left[\left(Y' \right)^r \right]=2^{\, r} \, \sum_{i=0}^{r} \frac{1}{i!}\left[\frac{d^i}{d h^i} \left(h\right)_r\right] \mathbb{E}\left(M^i\right),$$
where, by the general formula for the moments about zero of the Poisson distribution \cite{JohKemKot05}, one has:
\begin{equation}
\mathbb{E}\left(M^i\right)=\sum_{j=0}^{i} \mathcal{S}\left(i,j\right) \left(\frac{\lambda}{2}\right)^j, \qquad i \in \mathbb{N},
\label{eq:mom.pois}
\end{equation}
$\mathcal{S}\left(i,j\right)$ being a Stirling number of the second kind. Hence, Equation~(\ref{eq:mom.ncchisq}) is established.
\end{proof}
We can thus recover the raw moments of order $r=1,2,3,4$ of the $\chi'^{\,2}_g \left(\lambda \right)$ distribution as special cases of Equation~(\ref{eq:mom.ncchisq}):
\begin{equation}
\mathbb{E}\left(\chi'^{\,2}_g \left(\lambda \right)\right)=2\left[\mathcal{S}\left(0,0\right) \, h+\mathcal{S}\left(1,1\right) \left(\lambda/2\right)\right]=g+\lambda  \, ,
\label{eq:mom1.ncchisq}
\end{equation}
\begin{eqnarray}
\lefteqn{\mathbb{E}\left(\chi'^{\,2}_g \left(\lambda \right) \right)^2=} \nonumber \\
& = & 4\left\{\mathcal{S}\left(0,0\right) \, h\left(h+1\right)+\left[\, \mathcal{S}\left(1,1\right) \left(2h+1\right)+\mathcal{S}\left(2,1\right)\right]\left(\lambda/2\right)+\mathcal{S}\left(2,2\right)\left(\lambda/2\right)^2 \right\} \; \nonumber \\
& = & g\left(g+2\right)+2\left(g+2\right)\lambda+\lambda^2 \, ,
\label{eq:mom2.ncchisq}
\end{eqnarray}
\begin{eqnarray}
\lefteqn{\mathbb{E}\left(\chi'^{\,2}_g \left(\lambda \right) \right)^3=}\nonumber \\
& = & 8\left\{\mathcal{S}\left(0,0\right) h \left(h+1\right)\left(h+2\right)+\left[\mathcal{S}\left(1,1\right) \left(3h^2+6h+2\right)+3 \, \mathcal{S}\left(2,1\right) \left(h+1\right) \right. \right. \quad \nonumber \\
& + & \left. \left. \mathcal{S}\left(3,1\right)\right] \left(\lambda/2\right)+\left[3 \, \mathcal{S}\left(2,2\right) \left(h+1\right)+\mathcal{S}\left(3,2\right)\right]\left(\lambda/2\right)^2+\mathcal{S}\left(3,3\right)\left(\lambda/2\right)^3 \right\}= \nonumber \\
& = & g\left(g+2\right)\left(g+4\right)+3\left(g+2\right)\left(g+4\right)\lambda+3\left(g+4\right)\lambda^2+\lambda^3 \, , 
\label{eq:mom3.ncchisq}
\end{eqnarray}
\begin{eqnarray}
\lefteqn{\mathbb{E}\left(\chi'^{\,2}_g \left(\lambda \right) \right)^4=} \nonumber \\
& = & 16\left\{\mathcal{S}\left(0,0\right) \, h \left(h+1\right)\left(h+2\right)\left(h+3\right)+ \left[2 \, \mathcal{S}\left(1,1\right) \left(2h^3+9h^2+11h+3\right)+ \right. \right. \; \; \nonumber \\
& + & \left. \mathcal{S}\left(2,1\right) \, \left(6h^2+18h+11\right)+2 \, \mathcal{S}\left(3,1\right)\left(2h+3\right)+\mathcal{S}\left(4,1\right)\right]\left(\lambda/2\right)+\nonumber \\
& + & \left[\mathcal{S}\left(2,2\right)\left(6h^2+18h+11\right)+2 \, \mathcal{S}\left(3,2\right)\left(2h+3\right)+\mathcal{S}\left(4,2\right)\right]\left(\lambda/2\right)^2+ \nonumber \\
& + & \left. \left[2 \, \mathcal{S}\left(3,3\right)\left(2h+3\right)+\mathcal{S}\left(4,3\right)\right]\left(\lambda/2\right)^3+\mathcal{S}\left(4,4\right)\left(\lambda/2\right)^4\right\}=\nonumber \\
& = & g\left(g+2\right)\left(g+4\right)\left(g+6\right)+4\left(g+2\right)\left(g+4\right)\left(g+6\right)\lambda+6\left(g+4\right)\left(g+6\right)\lambda^2 \nonumber \\
& + & 4\left(g+6\right)\lambda^3+\lambda^4 \, .
\label{eq:mom4.ncchisq}
\end{eqnarray}
However, neither the moment formula available in the literature nor the above derived one apply in case of number of degrees of freedom equal to zero. As far as the computation of the raw moments of the Purely Non-central Chi-Squared distribution is concerned, the following formula can be used.
\begin{proposition}[Moments about zero of the $\chi'^{\,2}_0 \left(\lambda \right)$ distribution]
\label{prop:mom.ncchisq.zero}
For every $r \in \mathbb{N}$, the $r$-th moment about zero of the $\chi'^{\,2}_0 \left(\lambda \right)$ distribution can be written as
\begin{equation}
\mathbb{E}\left(\chi'^{\,2}_0 \left(\lambda \right) \right)^r =2^{\, r} \, \sum_{i=0}^{r}\sum_{j=0}^{i} \left|s\left(r,i\right)\right|\mathcal{S}\left(i,j\right) \left(\frac{\lambda}{2}\right)^j \, ,
\label{eq:mom.ncchisq.zero}
\end{equation}
where $\left|s\left(r,i\right)\right|$ is an unsigned Stirling number of the first kind and $\mathcal{S}\left(i,j\right)$ is a Stirling number of the second kind.
\end{proposition}
\begin{proof}
The proof ensues from Equation~(\ref{eq:mom.ncchisq.dim1}) by taking $h=0$ and remembering that $(M)_r=\sum_{i=0}^r \left|s\left(r,i\right)\right|$ $M^i$ where $\left|s\left(r,i\right)\right|=\left(-1\right)^{r-i} s\left(r,i\right)$ and $s\left(r,i\right)$ is a Stirling number of the first kind. By Equation~(\ref{eq:mom.pois}), Equation~(\ref{eq:mom.ncchisq.zero}) is thus established.
\end{proof}
The special cases of Equation~(\ref{eq:mom.ncchisq.zero}) where $r$ is set equal to each of the integers from 1 to 4 lead to the expressions of the first four moments of the $\chi'^{\,2}_0 \left(\lambda \right)$ distribution:
$$\mathbb{E}\left(\chi'^{\,2}_0 \left(\lambda \right) \right)=2\left|s\left(1,1\right)\right| \mathcal{S}\left(1,1\right)=\lambda  \, ,$$
\begin{eqnarray*}
\lefteqn{\mathbb{E}\left(\chi'^{\,2}_0 \left(\lambda \right) \right)^2=}\\
& = & 4\left\{\left[ \, \left|s\left(2,1\right)\right|\mathcal{S}\left(1,1\right)+\left|s\left(2,2\right)\right|\mathcal{S}\left(2,1\right)\right]\left(\lambda/2\right)+\left|s\left(2,2\right)\right|\mathcal{S}\left(2,2\right)\left(\lambda/2\right)^2\right\}= \; \; \\
& = & 4\lambda+\lambda^2  \, , 
\end{eqnarray*}
\begin{eqnarray*}
\lefteqn{\mathbb{E}\left(\chi'^{\,2}_0 \left(\lambda \right) \right)^3=}\\
& = & 8\left\{\left[\, \left|s\left(3,1\right)\right|\mathcal{S}\left(1,1\right)+\left|s\left(3,2\right)\right|\mathcal{S}\left(2,1\right)+\left|s\left(3,3\right)\right|\mathcal{S}\left(3,1\right)\right]\left(\lambda/2\right)+ \right. \\
& + & \left. \left[\, \left|s\left(3,2\right)\right|\mathcal{S}\left(2,2\right)+\left|s\left(3,3\right)\right|\mathcal{S}\left(3,2\right)\right]\left(\lambda/2\right)^2+\left|s\left(3,3\right)\right|\mathcal{S}\left(3,3\right)\left(\lambda/2\right)^3 \right\}=  \\
& = & 24\lambda+12\lambda^2+\lambda^3 \, , 
\end{eqnarray*}
\begin{eqnarray*}
\lefteqn{\mathbb{E}\left(\chi'^{\,2}_0 \left(\lambda \right) \right)^4=}\\
& = & 16\left\{\left[ \, \left|s\left(4,1\right)\right|\mathcal{S}\left(1,1\right)+\left|s\left(4,2\right)\right|\mathcal{S}\left(2,1\right)+\left|s\left(4,3\right)\right|\mathcal{S}\left(3,1\right)+\left|s\left(4,4\right)\right|\mathcal{S}\left(4,1\right)\right] \right. \cdot \\
& \cdot & \left(\lambda/2\right)+ \left[\, \left|s\left(4,2\right)\right|\mathcal{S}\left(2,2\right)+\left|s\left(4,3\right)\right|\mathcal{S}\left(3,2\right)+\left|s\left(4,4\right)\right|\mathcal{S}\left(4,2\right)\right]\left(\lambda/2\right)^2+ \\
& + & \left. \left[ \, \left|s\left(4,3\right)\right|\mathcal{S}\left(3,3\right)+\left|s\left(4,4\right)\right|\mathcal{S}\left(4,3\right) \right]\left(\lambda/2\right)^3+\left|s\left(4,4\right)\right|\mathcal{S}\left(4,4\right)\left(\lambda/2\right)^4 \right\}= \nonumber \\
& = & 192\lambda+144\lambda^2+24\lambda^3+\lambda^4 \, .
\end{eqnarray*}
Observe that the above formulas can be alternatively obtained by taking $g=0$ in Equations~(\ref{eq:mom1.ncchisq}), (\ref{eq:mom2.ncchisq}), (\ref{eq:mom3.ncchisq}), (\ref{eq:mom4.ncchisq}), respectively.

\section{On the moments of the Non-central Beta distributions}
\label{sec:dnc.beta.mom}

The main weakness of the Doubly Non-central Beta distribution lies in its poor tractability from a mathematical standpoint. Despite the growing variety of applications attracted by the present model in recent years (see, for example, \cite{KimSch98} and \cite{Sta15}), the aforementioned drawback of the DNcB distribution poses strong limitations on its use as a model for data consisting of proportions. In this regard, the following Proposition~\ref{propo:dncb.cond.ind} makes explicit a new approach to the study of this class which clears the ground for a deeper analysis of it. This approach, which is just mentioned in the proof of Proposition 7 in \cite{OngOrs15}, has its origin in the following realization: in view of the arguments recalled in Section~\ref{sec:introduc}, Property~\ref{prope:char.prop.chisq} is no longer valid in the non-central setting. Hence, an interesting generalization of the latter to the non-central framework is made clear herein. Specifically, in a suitable conditional form, a Doubly Non-central Beta random variable $X'$ is independent of the sum of the two Non-central Chi-Squareds $Y'_1$ and $Y'_2$ involved in its definition. More precisely, this relationship applies conditionally on the sum $M^+$ of the two Poisson random variables on which both $X'$ and $Y'_1+Y'_2$ depend. The distribution of $X'$ given $M^+$ is also clarified.   

\begin{proposition}[Conditional independence]
\label{propo:dncb.cond.ind}
Let $X' \sim \mbox{\normalfont{DNcB}}\left(\alpha_1,\alpha_2,\lambda_1,\lambda_2\right)$ and $Y'_i$, $i=1,2$, be independent $\chi'^{\,2}_{2\alpha_i}\left(\lambda_i\right)$ random variables, with $Y'^+=Y'_1+Y'_2$. Furthermore, let $M_i$, $i=1,2$, be independent Poisson random variables with means $\lambda_i \, / \, 2$ and $M^+=M_1+M_2$. Then:
\begin{itemize}
\item[i)] $X'$ and $Y'^+$ are conditionally independent given $M^+$
\item[ii)] the density of $X'$ given $M^+$ is:
\begin{eqnarray}
\lefteqn{f_{\left.X' \, \right|\, M^+}\left(x\right)=}\nonumber \\
& = & \sum_{i=0}^{M^+} \left[\mbox{\normalfont{Binomial}}\left(i;M^+,\frac{\lambda_1}{\lambda^+}\right) \, \mbox{\normalfont{Beta}}\left(x;\alpha_1+i,\alpha_2+M^+-i\right) \right] \, , \qquad \qquad
\label{eq:dncb.distr.cond.m}
\end{eqnarray}
\end{itemize}
where $\mbox{\normalfont{Binomial}}\left(i;M^+,\frac{\lambda_1}{\lambda^+}\right)$ is the probability mass function of the $\mbox{\normalfont{Binomial}}\left(M^+,\frac{\lambda_1}{\lambda^+}\right)$ distribution evaluated in $i=0,\ldots,M^+$.
\end{proposition}
\begin{proof}
For a proof see \cite{OngOrs15}, Proposition 7.
\end{proof}
The easiness of the conditional density in Equation~(\ref{eq:dncb.distr.cond.m}), being a simple mixture of Beta densities, enables to overcome the aforementioned limitations steming from the mathematical complexity of the DNcB density. Indeed, in light of Proposition~\ref{propo:dncb.cond.ind}, the proof of a new formula for the raw moments of the Doubly Non-central Beta model becomes considerably simplified. In this regard, by analogy with the form of the density in Equation~(\ref{eq:dens.beta.dnc}), the $r$-th moment about zero of the $\mbox{\normalfont{DNcB}}$ distribution can be stated as
\begin{eqnarray}
\lefteqn{\mathbb{E}\left[\mbox{\normalfont{DNcB}}^{\, r}\left(\alpha_1,\alpha_2,\lambda_1,\lambda_2\right)\right]=}\nonumber \\
& = & \sum_{j=0}^{+\infty} \sum_{k=0}^{+\infty}  \left\{\mbox{\normalfont{Poisson}}\left(j \, ; \, \frac{\lambda_1}{2}\right) \mbox{\normalfont{Poisson}}\left(k \, ; \, \frac{\lambda_2}{2}\right) \mathbb{E}\left[\mbox{\normalfont{Beta}}^{\, r}\left(\alpha_1+j,\alpha_2+k\right)\right]\right\} \, . \qquad 
\label{eq:momr.beta.dnc.doublesers}
\end{eqnarray}
By Equation~(\ref{eq:poch.symb.sum}), Equation~(\ref{eq:momr.beta.dnc.doublesers}) can be equivalently expressed as the following infinite sum of Kummer's confluent hypergeometric functions: 
\begin{eqnarray}
\lefteqn{\mathbb{E}\left[\mbox{\normalfont{DNcB}}^{\, r}\left(\alpha_1,\alpha_2,\lambda_1,\lambda_2\right)\right]=\mathbb{E}\left[\mbox{\normalfont{Beta}}^{\, r}\left(\alpha_1,\alpha_2\right)\right] \cdot}\nonumber \\
& \cdot & e^{-\frac{\lambda^+}{2}} \sum_{j=0}^{+\infty}   \frac{\left(\alpha^+\right)_j}{\left(\alpha_1\right)_j} \frac{\left(\alpha_1+r\right)_j}{\left(\alpha^++r\right)_j}  \frac{\left(\frac{\lambda_1}{2}\right)^j}{j!} \, _1F_1\left(\alpha^++j;\alpha^++r+j;\frac{\lambda_2}{2}\right)  \, ; \qquad \qquad \,
\label{eq:momr.beta.dnc.oneser}
\end{eqnarray}
however, this formula is computationally cumbersome. That said, a new general formula for the raw moments of the $\mbox{DNcB}$ distribution is provided by the following Proposition, which broadens and completes Proposition 7 in \cite{OngOrs15}. This formula allows the computation of the quantity at study to be reduced from the single infinite series in Equation~(\ref{eq:momr.beta.dnc.oneser}) to a surprisingly simple form given by the following finite sum.

\begin{proposition}[Moments about zero of the $\mbox{\normalfont{DNcB}}$ distribution]
\label{propo:momr.beta.dnc.rapp}
\begin{eqnarray}
\lefteqn{\mathbb{E}\left[\mbox{\normalfont{DNcB}}^{\, r}\left(\alpha_1,\alpha_2,\lambda_1,\lambda_2\right)\right]= \qquad \qquad \qquad \qquad \qquad \qquad r \in \mathbb{N}}\nonumber \\
& = & \mathbb{E}\left[\mbox{\normalfont{Beta}}^{\, r}\left(\alpha_1,\alpha_2\right)\right] \cdot e^{-\frac{\lambda^+}{2}}\sum_{i=0}^{r} \frac{{r \choose i}\left(\alpha^+\right)_i\left(\frac{\lambda_1}{2}\right)^i}{\left(\alpha_1\right)_i \left(\alpha^++r\right)_i} \, _1F_1\left(\alpha^++i;\alpha^++r+i;\frac{\lambda^+}{2}\right) \, . \nonumber \\
\label{eq:momr.beta.dnc}
\end{eqnarray}
\end{proposition}
\begin{proof}
Let $X' \sim \mbox{\normalfont{DNcB}}\left(\alpha_1,\alpha_2,\lambda_1,\lambda_2\right)$ and $L$ have a $\mbox{Binomial}\left(M^+,\theta_1\right)$ distribution conditionally on $M^+ \sim \mbox{Poisson}\left(\lambda^+ \, / \, 2\right)$ with $\theta_1=\lambda_1 \, / \, \lambda^+$. By virtue of Equation~(\ref{eq:dncb.distr.cond.m}), one has:
\begin{eqnarray}
\lefteqn{\mathbb{E}\left[\left. \left(X'\right)^r \right| \, M^+\right]=}\nonumber \\
& = & \int_0^1 x^{\, r} \, f_{\left.X' \, \right|\, M^+}\left(x\right) \, dx=\sum_{i=0}^{M^+} \frac{\left(\alpha_1+i\right)_r}{\left(\alpha^++M^+\right)_r} \, {M^+ \choose i} \theta_1^{\, i} \left(1-\theta_1\right)^{M^+- \, i}= \nonumber \\
& = & \frac{\mathbb{E}\left[\left. \left(\alpha_1+L\right)_r \right| \, M^+\right]}{\left(\alpha^++M^+\right)_r} \, ,
\label{eq:momr.beta.dnc.dim1}
\end{eqnarray}
where, in light of Equation~(\ref{eq:poch.symb.binom}):
$$\left(\alpha_1+L\right)_r=\left[\left(\alpha_1-1\right)+\left(L+1\right)\right]_r=\sum_{i=0}^{r} {r \choose i} \left(\alpha_1-1\right)_{r-i} \left(L+1\right)_i \, ,$$
so that:
\begin{eqnarray}
\lefteqn{\mathbb{E}\left[\left. \left(\alpha_1+L\right)_r \right| \, M^+\right]=}\nonumber \\
& = & \sum_{i=0}^r {r \choose i} \left(\alpha_1-1\right)_{r-i} \, \mathbb{E}\left[\left. \left(L+1\right)_i \, \right| \, M^+\right]=\nonumber \\
& = & \sum_{i=0}^{r} {r \choose i} \left(\alpha_1-1\right)_{r-i} \sum_{l=0}^{M^+} \left(l+1\right)_i {M^+ \choose l} \left(1-\theta_1\right)^{M^+- \, l} \theta_1^{\, l} \, .
\label{eq:momr.beta.dnc.dim2}
\end{eqnarray}
By Equation~(\ref{eq:poch.symb}), for every $i=0,\ldots,r \,$:
\begin{equation}
\left(l+1\right)_i=\frac{\Gamma\left(l+i+1\right)}{\Gamma\left(l+1\right)}=\frac{\left(l+i\right) !}{l !}={l+i \choose l} \, i! \; ;
\label{eq:momr.beta.dnc.dim3}
\end{equation}
under Equation~(\ref{eq:momr.beta.dnc.dim3}), Equation~(\ref{eq:momr.beta.dnc.dim2}) can be thus rewritten as follows:
\begin{equation}
\mathbb{E}\left[\left. \left(\alpha_1+L\right)_r \right| \, M^+\right]=\sum_{i=0}^{r} {r \choose i} i ! \left(\alpha_1-1\right)_{r-i} \sum_{l=0}^{M^+} {l+i \choose l} {M^+ \choose l} \left(1-\theta_1\right)^{M^+- \, l} \theta_1^{\, l} \, .
\label{eq:momr.beta.dnc.dim4}
\end{equation}
In carrying out the prove, reference must be made to Ljunggren's Identity, namely
\begin{equation}
\sum_{k=0}^{n} {\alpha+k \choose k} {n \choose k} \left(x-y\right)^{n-k} y^k=\sum_{k=0}^{n} {\alpha \choose k} {n \choose k} x^{n-k} y^k \, , 
\label{eq:ljunggren.id}
\end{equation}
which is (3.18) in \cite{Gou72}. By setting $k=l$, $n=M^+$, $\alpha=i$, $x=1$, $y=\theta_1$ in Equation~(\ref{eq:ljunggren.id}), Equation~(\ref{eq:momr.beta.dnc.dim4}) can be restated in the following form:
\begin{equation}
\mathbb{E}\left[\left. \left(\alpha_1+L\right)_r \right| \, M^+\right]=\sum_{i=0}^{r} {r \choose i} i ! \left(\alpha_1-1\right)_{r-i} \sum_{l=0}^{M^+} {i \choose l} {M^+ \choose l} \theta_1^{\, l} \, ,
\label{eq:momr.beta.dnc.dim5}
\end{equation}
so that, under Equation~(\ref{eq:momr.beta.dnc.dim5}), Equation~(\ref{eq:momr.beta.dnc.dim1}) can be equivalently expressed as
$$\mathbb{E}\left[\left. \left(X'\right)^r \right| \, M^+\right]=\frac{1}{\left(\alpha^++M^+\right)_r} \, \sum_{i=0}^{r} {r \choose i} i ! \left(\alpha_1-1\right)_{r-i} \sum_{l=0}^{M^+} {i \choose l} {M^+ \choose l} \theta_1^{\, l} \, .$$
By the law of iterated expectations one has:
$$\mathbb{E}\left[\left(X'\right)^r\right]=e^{-\frac{\lambda^+}{2}} \, \sum_{i=0}^{r} {r \choose i} i ! \left(\alpha_1-1\right)_{r-i} \sum_{m=0}^{+\infty} \frac{\left(\frac{\lambda^+}{2}\right)^m}{m ! \left(\alpha^++m\right)_r} \, \sum_{l=0}^{m} {i \choose l} {m \choose l} \theta_1^{\, l} \, ,$$
where, by Equations~(\ref{eq:poch.symb}) and~(\ref{eq:poch.symb.sum}), for every $m \in \mathbb{N} \cup \{0\}$:
$$\frac{\left(\frac{\lambda^+}{2}\right)^m}{m ! \left(\alpha^++m\right)_r}=\frac{\left(\frac{\lambda^+}{2}\right)^m \left(\alpha^+\right)_m}{m ! \left(\alpha^+\right)_{r+m}}=\frac{\left(\frac{\lambda^+}{2}\right)^m \Gamma\left(\alpha^++m\right)}{m ! \; \Gamma\left(\alpha^++r+m\right)} \, ,$$
so that:
$$\mathbb{E}\left[\left(X'\right)^r\right]=e^{-\frac{\lambda^+}{2}} \, \sum_{i=0}^{r} {r \choose i} i ! \left(\alpha_1-1\right)_{r-i} \sum_{m=0}^{+\infty} \frac{\left(\frac{\lambda^+}{2}\right)^m \Gamma\left(\alpha^++m\right)}{m ! \; \Gamma\left(\alpha^++r+m\right)} \, \sum_{l=0}^{m} {i \choose l} {m \choose l} \theta_1^{\, l} \, .$$
By bearing in mind that ${i \choose l}=0$ for $l>i$, we have:
$$\mathbb{E}\left[\left(X'\right)^r\right]=e^{-\frac{\lambda^+}{2}} \, \sum_{i=0}^{r} {r \choose i} i ! \left(\alpha_1-1\right)_{r-i} \sum_{l=0}^{i} \frac{\theta_1^{\, l} }{l !} {i \choose l} \sum_{m=l}^{+\infty} \frac{\left(\frac{\lambda^+}{2}\right)^m \Gamma\left(\alpha^++m\right)}{\left(m-l\right) ! \; \Gamma\left(\alpha^++r+m\right)} \, ;$$
by setting $k=m-l \Leftrightarrow m=l+k$, by virtue of Equation~(\ref{eq:poch.symb}) and in light of Equation~(\ref{eq:f11}):
\begin{eqnarray}
\lefteqn{\mathbb{E}\left[\left(X'\right)^r\right]=} \nonumber \\
& = & e^{-\frac{\lambda^+}{2}} \, \sum_{i=0}^{r} {r \choose i} i ! \left(\alpha_1-1\right)_{r-i} \sum_{l=0}^{i} \frac{\left(\frac{\lambda_1}{2}\right)^l \Gamma\left(\alpha^++l\right) }{l ! \; \Gamma\left(\alpha^++r+l\right)} {i \choose l} \sum_{k=0}^{+\infty} \frac{\left(\frac{\lambda^+}{2}\right)^k \left(\alpha^++l\right)_k}{k ! \, \left(\alpha^++r+l\right)_k}= \nonumber \\
& = & e^{-\frac{\lambda^+}{2}} \cdot \nonumber \\
& \cdot & \sum_{i=0}^{r} {r \choose i} i ! \left(\alpha_1-1\right)_{r-i} \sum_{l=0}^{i} \frac{\left(\frac{\lambda_1}{2}\right)^l \Gamma\left(\alpha^++l\right) }{l ! \; \Gamma\left(\alpha^++r+l\right)} {i \choose l} \, _1F_1\left(\alpha^++l;\alpha^++r+l;\frac{\lambda^+}{2}\right)=\nonumber \\
& = & \frac{e^{-\frac{\lambda^+}{2}}}{\left(\alpha^+\right)_r} \cdot \nonumber \\
& \cdot & \sum_{l=0}^{r} \frac{\left(\frac{\lambda_1}{2}\right)^l \left(\alpha^+\right)_l }{\left(l !\right)^2 \left(\alpha^++r\right)_l}  \, _1F_1\left(\alpha^++l;\alpha^++r+l;\frac{\lambda^+}{2}\right) \, \sum_{i=l}^{r} {r \choose i} \left(\alpha_1-1\right)_{r-i} \frac{\left(i !\right)^2}{\left(i-l\right)!} \, . \quad \nonumber \\
\label{eq:momr.beta.dnc.dim7}
\end{eqnarray}
For every $l=0,\ldots,r$, by setting $j=i-l \Leftrightarrow i=j+l$ and by virtue of Equation~(\ref{eq:poch.symb.binom}), the final sum in Equation~(\ref{eq:momr.beta.dnc.dim7}) is tantamount to:
\begin{eqnarray}
\lefteqn{\sum_{i=l}^{r} {r \choose i} \left(\alpha_1-1\right)_{r-i} \frac{\left(i!\right)^2}{\left(i-l\right)!}=} \nonumber \\
& = & \sum_{j=0}^{r-l} {r \choose j+l} \left(\alpha_1-1\right)_{r-j-l} \frac{\left[\left(j+l\right)!\right]^2}{j!}=\frac{r!}{\left(r-l\right)!}\sum_{j=0}^{r-l} {r-l \choose j} \left(\alpha_1-1\right)_{r-j-l} \, \left(j+l\right)! \nonumber \\
& = & \frac{r! \, l!}{\left(r-l\right)!}\sum_{j=0}^{r-l} {r-l \choose j} \left(\alpha_1-1\right)_{r-j-l} \, \left(l+1\right)_j=\frac{r! \, l!}{\left(r-l\right)!} \left[\left(\alpha_1-1\right)+\left(l+1\right)\right]_{r-l}=\nonumber \\
& = & \frac{r! \, l!}{\left(r-l\right)!} \left(\alpha_1+l\right)_{r-l} \, . \quad \quad
\label{eq:momr.beta.dnc.dim8}
\end{eqnarray}
Finally, under Equation~(\ref{eq:momr.beta.dnc.dim8}) and in light of Equations~(\ref{eq:poch.symb.ratio}) and~(\ref{eq:momr.beta}), Equation~(\ref{eq:momr.beta.dnc.dim7}) can be rewritten in the same form as in Equation~(\ref{eq:momr.beta.dnc}).
\end{proof}

The first moment of the $\mbox{DNcB}$ distribution can be thus obtained by taking $r=1$ in Equations~(\ref{eq:momr.beta}) and~(\ref{eq:momr.beta.dnc}) as follows:
\begin{eqnarray}
\lefteqn{\mathbb{E}\left[\mbox{\normalfont{DNcB}}\left(\alpha_1,\alpha_2,\lambda_1,\lambda_2\right)\right]=} \nonumber \\
& = & \frac{\alpha_1}{\alpha^+}\, e^{-\frac{\lambda^+}{2}} \left[_1F_1\left(\alpha^+;\alpha^++1;\frac{\lambda^+}{2}\right)+\frac{\alpha^+ \, \frac{\lambda_1}{2}}{\alpha_1 \left(\alpha^++1\right)} \, _1F_1\left(\alpha^++1;\alpha^++2;\frac{\lambda^+}{2}\right)\right] \, . \nonumber \\
\label{eq:mom1.beta.dnc}
\end{eqnarray}
The expression in Equation~(\ref{eq:mom1.beta.dnc}) can be algebraically manipulated with the aim to reduce the number of distinct $_1F_1$ functions from two to one. Specifically, this latter formula of the $\mbox{DNcB}\left(\alpha_1,\alpha_2,\lambda_1,\lambda_2\right)$ mean can be rearranged so as to yield the following convex linear combination between the $\mbox{Beta}\left(\alpha_1,\alpha_2\right)$ mean and the compositional ratio of the non-centrality parameters with respect to $\lambda_1$. These two additive components are given weights depending only on one $_1F_1$ function. 
\begin{proposition}[Alternative expression of the $\mbox{\normalfont{DNcB}}$ mean]
\label{propo:mom1.beta.dnc.improv}
\begin{eqnarray}
\lefteqn{\mathbb{E}\left[\mbox{\normalfont{DNcB}}\left(\alpha_1,\alpha_2,\lambda_1,\lambda_2\right)\right]=} \nonumber \\
& = & \frac{\alpha_1}{\alpha^+} \left[\, e^{-\frac{\lambda^+}{2}} \, _1F_1\left(\alpha^+;\alpha^++1;\frac{\lambda^+}{2}\right)\right]+\frac{\lambda_1}{\lambda^+} \left[1-e^{-\frac{\lambda^+}{2}} \, _1F_1\left(\alpha^+;\alpha^++1;\frac{\lambda^+}{2}\right)\right] \, . \nonumber \\
\label{eq:mom1.beta.dnc.improv}
\end{eqnarray}
\end{proposition}
\begin{proof}
By taking $a=\alpha^++1$, $b=\alpha^++2$, $x=\lambda^+/ \, 2$ in Equation~(\ref{eq:f11.rec.rel1}), or, equivalently, in Equation~(\ref{eq:f11.rec.rel4}), one obtains:
\begin{equation}
_1F_1\left(\alpha^+;\alpha^++2;\frac{\lambda^+}{2}\right)=\left(\alpha^++1\right) e^{\frac{\lambda^+}{2}}-\left(\alpha^++\frac{\lambda^+}{2}\right) \, _1F_1\left(\alpha^++1;\alpha^++2;\frac{\lambda^+}{2}\right) \, , \;
\label{eq:mom1.beta.dnc.improv.dim1}
\end{equation}
while the special case of Equation~(\ref{eq:f11.rec.rel2}), or, equivalently, of Equation~(\ref{eq:f11.rec.rel3}), corresponding to $a=\alpha^+$, $b=\alpha^++1$, $x=\lambda^+/ \, 2$ leads, under Equation~(\ref{eq:mom1.beta.dnc.improv.dim1}), by simple computations, to:
\begin{equation}
_1F_1\left(\alpha^++1;\alpha^++2;\frac{\lambda^+}{2}\right)=\frac{\alpha^++1}{\frac{\lambda^+}{2}} \left[ \, e^{\frac{\lambda^+}{2}}- \, _1F_1\left(\alpha^+;\alpha^++1;\frac{\lambda^+}{2}\right)\right] \, .
\label{eq:mom1.beta.dnc.improv.dim2}
\end{equation}
Finally, under Equation~(\ref{eq:mom1.beta.dnc.improv.dim2}), Equation~(\ref{eq:mom1.beta.dnc}) can be exhibited in the form of Equation~(\ref{eq:mom1.beta.dnc.improv}).
\end{proof}
Analogously, the formula of the second moment of the $\mbox{DNcB}$ distribution ensues from Equations~(\ref{eq:momr.beta}) and~(\ref{eq:momr.beta.dnc}) by setting $r=2$:
\begin{eqnarray}
\lefteqn{\mathbb{E}\left[\mbox{\normalfont{DNcB}}^{\, 2}\left(\alpha_1,\alpha_2,\lambda_1,\lambda_2\right)\right]=} \nonumber \\
& = & \frac{\left(\alpha_1\right)_2}{\left(\alpha^+\right)_2} \, e^{-\frac{\lambda^+}{2}} \left[_1F_1\left(\alpha^+;\alpha^++2;\frac{\lambda^+}{2}\right) + \frac{\alpha^+ \, \lambda_1}{\alpha_1 \left(\alpha^++2\right)} \,  _1F_1\left(\alpha^++1;\alpha^++3;\frac{\lambda^+}{2}\right)+ \right. \nonumber \\
& + & \left. \frac{\left(\alpha^+\right)_2 \, \left(\lambda_1/2\right)^2}{\left(\alpha_1\right)_2 \left(\alpha^++2\right)_2} \,  _1F_1\left(\alpha^++2;\alpha^++4;\frac{\lambda^+}{2}\right)\right] \, .
\label{eq:mom2.beta.dnc}
\end{eqnarray}
A restatement of Equation~(\ref{eq:mom2.beta.dnc}) depending only on two $_1F_1$ functions instead of three is derived herein.
\begin{proposition}[Alternative expression of the second moment of the $\mbox{\normalfont{DNcB}}$ distribution]
\label{propo:mom2.beta.dnc.improv}
\begin{eqnarray}
\lefteqn{\mathbb{E}\left[\mbox{\normalfont{DNcB}}^{\, 2}\left(\alpha_1,\alpha_2,\lambda_1,\lambda_2\right)\right]=\frac{\left(\alpha_1\right)_2}{\left(\alpha^+\right)_2} \, e^{-\frac{\lambda^+}{2}} \, _1F_1\left(\alpha^+;\alpha^++2;\frac{\lambda^+}{2}\right)+} \nonumber \\
& + & \frac{1}{\alpha^++2} \left[\frac{\lambda_1\left(\alpha_1+1\right)}{\alpha^++1}-\frac{\left(\lambda_1/2\right)^2}{\alpha^++1+\frac{\lambda^+}{2}} \right] \, e^{-\frac{\lambda^+}{2}} \, _1F_1\left(\alpha^++1;\alpha^++3;\frac{\lambda^+}{2}\right)+ \nonumber \\
& + & \frac{\left(\lambda_1/2\right)^2}{\frac{\lambda^+}{2}\left(\alpha^++1+\frac{\lambda^+}{2}\right)} \left[1-e^{-\frac{\lambda^+}{2}} \, _1F_1\left(\alpha^++1;\alpha^++3;\frac{\lambda^+}{2}\right)\right] \, . \qquad \qquad
\label{eq:mom2.beta.dnc.improv}
\end{eqnarray}
\end{proposition}
\begin{proof}
By taking $a=\alpha^++2$, $b=\alpha^++4$, $x=\lambda^+/ \, 2$ in Equation~(\ref{eq:f11.rec.rel4}), one obtains:
\begin{eqnarray}
\lefteqn{2 \, _1F_1\left(\alpha^++1;\alpha^++4;\frac{\lambda^+}{2}\right)=\left(\alpha^++3\right) \cdot} \nonumber \\
& \cdot &   _1F_1\left(\alpha^++2;\alpha^++3;\frac{\lambda^+}{2}\right) -\left(\alpha^++1+\frac{\lambda^+}{2}\right) \, _1F_1\left(\alpha^++2;\alpha^++4;\frac{\lambda^+}{2}\right) \, , \quad \quad
\label{eq:mom2.beta.dnc.improv.dim1}
\end{eqnarray}
while the special case of Equation~(\ref{eq:f11.rec.rel3}) where $a=\alpha^++1$, $b=\alpha^++3$, $x=\lambda^+/ \, 2$ leads to:
\begin{eqnarray}
\lefteqn{\left(\alpha^++3\right)\left(\alpha^++1+\frac{\lambda^+}{2}\right) \, _1F_1\left(\alpha^++1;\alpha^++3;\frac{\lambda^+}{2}\right)=\left(\alpha^++1\right) \cdot} \nonumber \\
& \cdot & \left(\alpha^++3\right) \, _1F_1\left(\alpha^++2;\alpha^++3;\frac{\lambda^+}{2}\right) +2 \cdot \frac{\lambda^+}{2} \,  _1F_1\left(\alpha^++1;\alpha^++4;\frac{\lambda^+}{2}\right) \, ; \quad
\label{eq:mom2.beta.dnc.improv.dim2}
\end{eqnarray}
hence, under Equation~(\ref{eq:mom2.beta.dnc.improv.dim1}), Equation~(\ref{eq:mom2.beta.dnc.improv.dim2}) reduces to:
\begin{eqnarray}
\lefteqn{\left(\alpha^++1+\frac{\lambda^+}{2}\right)\left\{\left(\alpha^++3\right) \left[\, _1F_1\left(\alpha^++1;\alpha^++3;\frac{\lambda^+}{2}\right) \right. \right. +}\nonumber \\
& - & \left. \left. \, _1F_1\left(\alpha^++2;\alpha^++3;\frac{\lambda^+}{2}\right)\right]+\frac{\lambda^+}{2} \, _1F_1\left(\alpha^++2;\alpha^++4;\frac{\lambda^+}{2}\right)\right\}=0 \, .
\label{eq:mom2.beta.dnc.improv.dim3}
\end{eqnarray}
By setting $a=\alpha^++2$, $b=\alpha^++3$, $x=\lambda^+/ \, 2$ in Equation~(\ref{eq:f11.rec.rel4}), or, equivalently, in Equation~(\ref{eq:f11.rec.rel1}), one has:
\begin{eqnarray}
\lefteqn{\left(\alpha^++1+\frac{\lambda^+}{2}\right) \,  _1F_1\left(\alpha^++2;\alpha^++3;\frac{\lambda^+}{2}\right) \qquad =} \nonumber \\
& = & \qquad \left(\alpha^++2\right) \, e^{\frac{\lambda^+}{2}} - \, _1F_1\left(\alpha^++1;\alpha^++3;\frac{\lambda^+}{2}\right) \, ,
\label{eq:mom2.beta.dnc.improv.dim4}
\end{eqnarray}
so that, under Equation~(\ref{eq:mom2.beta.dnc.improv.dim4}), Equation~(\ref{eq:mom2.beta.dnc.improv.dim3}) reduces to:
\begin{eqnarray}
\lefteqn{_1F_1\left(\alpha^++2;\alpha^++4;\frac{\lambda^+}{2}\right)=\frac{\alpha^++3}{\frac{\lambda^+}{2}\left(\alpha^++1+\frac{\lambda^+}{2}\right)} \cdot}\nonumber \\
& \cdot & \left[\left(\alpha^++2\right) \, e^{\frac{\lambda^+}{2}}-\left(\alpha^++2+\frac{\lambda^+}{2}\right) \,  _1F_1\left(\alpha^++1;\alpha^++3;\frac{\lambda^+}{2}\right)\right] \, .
\label{eq:mom2.beta.dnc.improv.dim5}  
\end{eqnarray}
Finally, Equation~(\ref{eq:mom2.beta.dnc.improv}) is obtained from Equation~(\ref{eq:mom2.beta.dnc}) with simple computations by making use of Equation~(\ref{eq:mom2.beta.dnc.improv.dim5}).
\end{proof}
As far as the computation of improved expressions for higher order moments of the $\mbox{DNcB}$ distribution is concerned, similar reasonings apply; however, other recurrence relations are to be used in order to reduce the number of distinct $_1F_1$ functions (in this regard see \cite{AbrSte64}) and more and more burdensome algebraic manipulations are needed.

In light of Equation~(\ref{eq:momr.beta.dnc}), a new general formula for the moments about zero of the $\mbox{\normalfont{NcB1}}$ distribution can be derived regardless of Equation~(\ref{eq:momr.beta.nc1.def}).
\begin{proposition}[Moments about zero of the $\mbox{\normalfont{NcB1}}$ distribution]
\label{propo:momr.beta.nc1}
\begin{eqnarray}
\lefteqn{\mathbb{E}\left[\mbox{\normalfont{NcB1}}^{\, r}\left(\alpha_1,\alpha_2,\lambda\right)\right]= \qquad \qquad \qquad \qquad \qquad \qquad r \in \mathbb{N}}	\nonumber \\
& = & \mathbb{E}\left[\mbox{\normalfont{Beta}}^{\, r}\left(\alpha_1,\alpha_2\right)\right] \, e^{-\frac{\lambda}{2}}\sum_{i=0}^{r} \frac{{r \choose i}\left(\alpha^+\right)_i\left(\frac{\lambda}{2}\right)^i}{\left(\alpha_1\right)_i \left(\alpha^++r\right)_i} \, _1F_1\left(\alpha^++i;\alpha^++r+i;\frac{\lambda}{2}\right) \, . \nonumber \\
\label{eq:momr.beta.nc1}
\end{eqnarray}
\end{proposition}
\begin{proof}
The proof of Equation~(\ref{eq:momr.beta.nc1}) follows from Equation~(\ref{eq:momr.beta.dnc}) by taking $\lambda_2=0$ and renaming $\lambda_1$ in $\lambda$.
\end{proof}

By comparing Equations~(\ref{eq:momr.beta.nc1.def}) and~(\ref{eq:momr.beta.nc1}), the following identity holds true for the hypergeometric functions involved in the two aforementioned moment formulas.

\begin{proposition}[Identity]
\label{propo:id.hyperg.funct}
Let $b>a>0$, $n \in \mathbb{N}$ and $x>0$. Then:
\begin{equation}
_2F_2\left(a+n,b;a,b+n;x\right)=\sum_{i=0}^{n} \frac{{n \choose i}\left(b\right)_i \, x^i}{\left(a\right)_i \left(b+n\right)_i} \, _1F_1\left(b+i;b+n+i;x\right).
\label{eq:id.hyperg.funct}
\end{equation}
\end{proposition}
\begin{proof}
Equating the right-hand sides of Equations~(\ref{eq:momr.beta.nc1.def}) and~(\ref{eq:momr.beta.nc1}) and setting $\alpha_1=a$, $\alpha^+=b\left(>a\right)$, $r=n$ and $\frac{\lambda}{2}=x$ yield Equation~(\ref{eq:id.hyperg.funct}).
\end{proof}

Moreover, the general formula for the moments about zero of the NcB2 distribution in Equation~(\ref{eq:momr.beta.nc2}) can be recovered as the special case of Equation~(\ref{eq:momr.beta.dnc}) where $\lambda_1$ is set equal to 0 and $\lambda_2$ is renamed in $\lambda$. Finally, an interesting relationship applies among the means of the above mentioned types of Non-central Beta distributions.   

\begin{proposition}[Relationship among the means of the $\mbox{\normalfont{DNcB}}$, the $\mbox{\normalfont{NcB1}}$ and the $\mbox{\normalfont{NcB2}}$ distributions]
$$\mathbb{E}\left[\mbox{\normalfont{DNcB}}\left(\alpha_1,\alpha_2,\lambda_1,\lambda_2\right)\right]=\frac{\lambda_1}{\lambda^+} \, \mathbb{E}\left[\mbox{\normalfont{NcB1}}\left(\alpha_1,\alpha_2,\lambda^+\right)\right]+\frac{\lambda_2}{\lambda^+} \, \mathbb{E}\left[\mbox{\normalfont{NcB2}}\left(\alpha_1,\alpha_2,\lambda^+\right)\right] \, .$$
\end{proposition}
\begin{proof}
The proof is straightforward once we observe that Equation~(\ref{eq:mom1.beta.dnc}) can be rewritten as follows:
\begin{eqnarray*}
\lefteqn{\mathbb{E}\left[\mbox{\normalfont{DNcB}}\left(\alpha_1,\alpha_2,\lambda_1,\lambda_2\right)\right]=}\\
& = & \frac{\alpha_1}{\alpha^+}\, e^{-\frac{\lambda^+}{2}} \, _1F_1\left(\alpha^+;\alpha^++1;\frac{\lambda^+}{2}\right)+e^{-\frac{\lambda^+}{2}} \frac{\frac{\lambda_1}{2}}{\alpha^++1} \, _1F_1\left(\alpha^++1;\alpha^++2;\frac{\lambda^+}{2}\right)=\\
& = & \frac{\lambda_1}{\lambda^+} \left[\frac{\alpha_1}{\alpha^+}\, e^{-\frac{\lambda^+}{2}} \, _1F_1\left(\alpha^+;\alpha^++1;\frac{\lambda^+}{2}\right)+e^{-\frac{\lambda^+}{2}} \frac{\frac{\lambda^+}{2}}{\alpha^++1} \, _1F_1\left(\alpha^++1;\alpha^++2;\frac{\lambda^+}{2}\right)\right]+\\
& + & \frac{\lambda_2}{\lambda^+}\left[\frac{\alpha_1}{\alpha^+}\, e^{-\frac{\lambda^+}{2}} \, _1F_1\left(\alpha^+;\alpha^++1;\frac{\lambda^+}{2}\right)\right] \, .
\end{eqnarray*}
\end{proof}

\section{Simulation results}
\label{sec:simul_res}

A simulation study aimed at confirming the validity of the derived moment formulas is clearly needed. This makes it necessary to generate from the two distributions at study. Specifically, the issue of simulating from the Non-central Chi-Squared and the Doubly Non-central Beta models is addressed herein by resorting to the mixture representation of the former in Equation~(\ref{eq:mixrepres.ncchisq}) and to the conditional density of the latter given $M^+$ in Equation~(\ref{eq:dncb.distr.cond.m}). More precisely, the algorithm to obtain a realization from $X' \sim \mbox{\normalfont{DNcB}}\left(\alpha_1,\alpha_2,\lambda_1,\lambda_2\right)$ using the aforementioned method requires to first generate from the random variable $M^+ \sim \mbox{Poisson}\left(\lambda^+/ \, 2\right)$ and then from the mixture of $M^++1$ Beta distributions referred to hereinabove. To sample from this latter mixture, one chooses an index $i^*$ from $\{0,\ldots,M^+\}$ according to the probabilities of the $\mbox{Binomial}\left(M^+,\lambda_1 \, / \, \lambda^+\right)$ distribution and then simulates a value from the corresponding $\mbox{Beta}\left(\alpha_1+i^*,\alpha_2+M^+-i^*\right)$ distribution. A graphical depiction of how the present algorithm works is displayed in Figure~\ref{fig:SIMUL}, which shows the histograms of the values simulated  from the $\mbox{DNcB}$ model together with the true densities of the latter for selected values of the shape parameters and the non-centrality parameters.

\begin{figure}[ht]
 \centering
 \subfigure
   {\includegraphics[width=6.6cm]{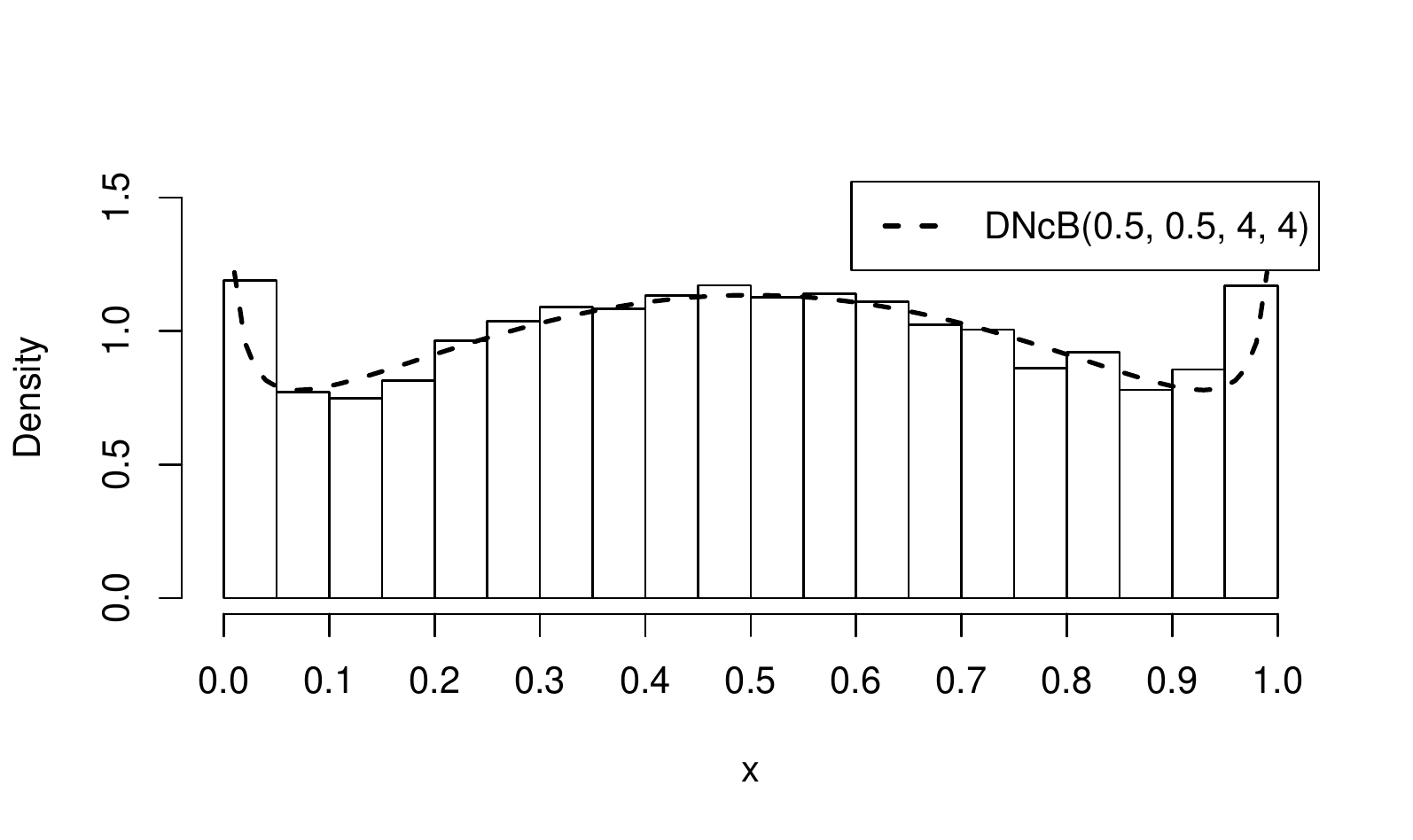}}
 \subfigure
   {\includegraphics[width=6.6cm]{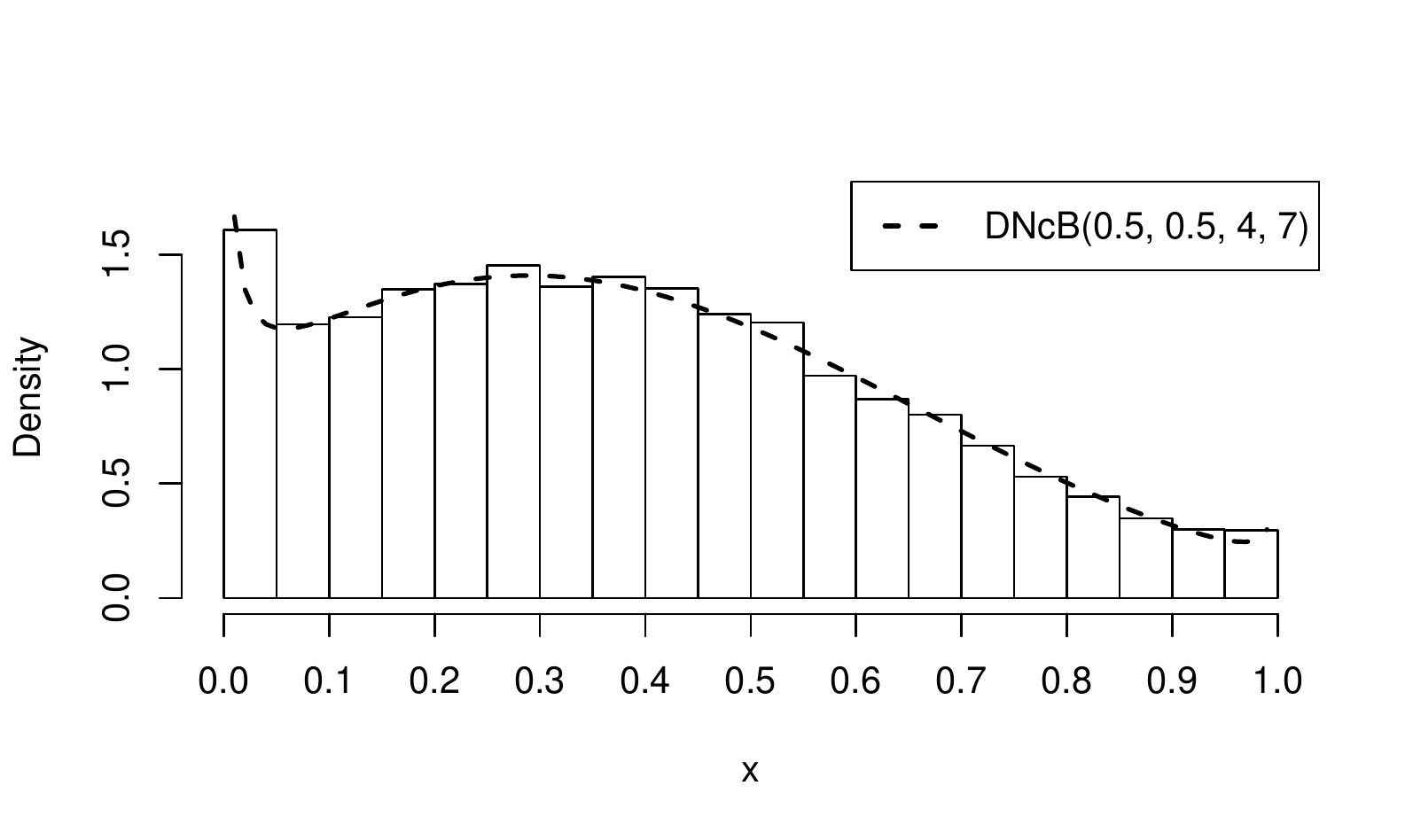}}
 \subfigure
   {\includegraphics[width=6.6cm]{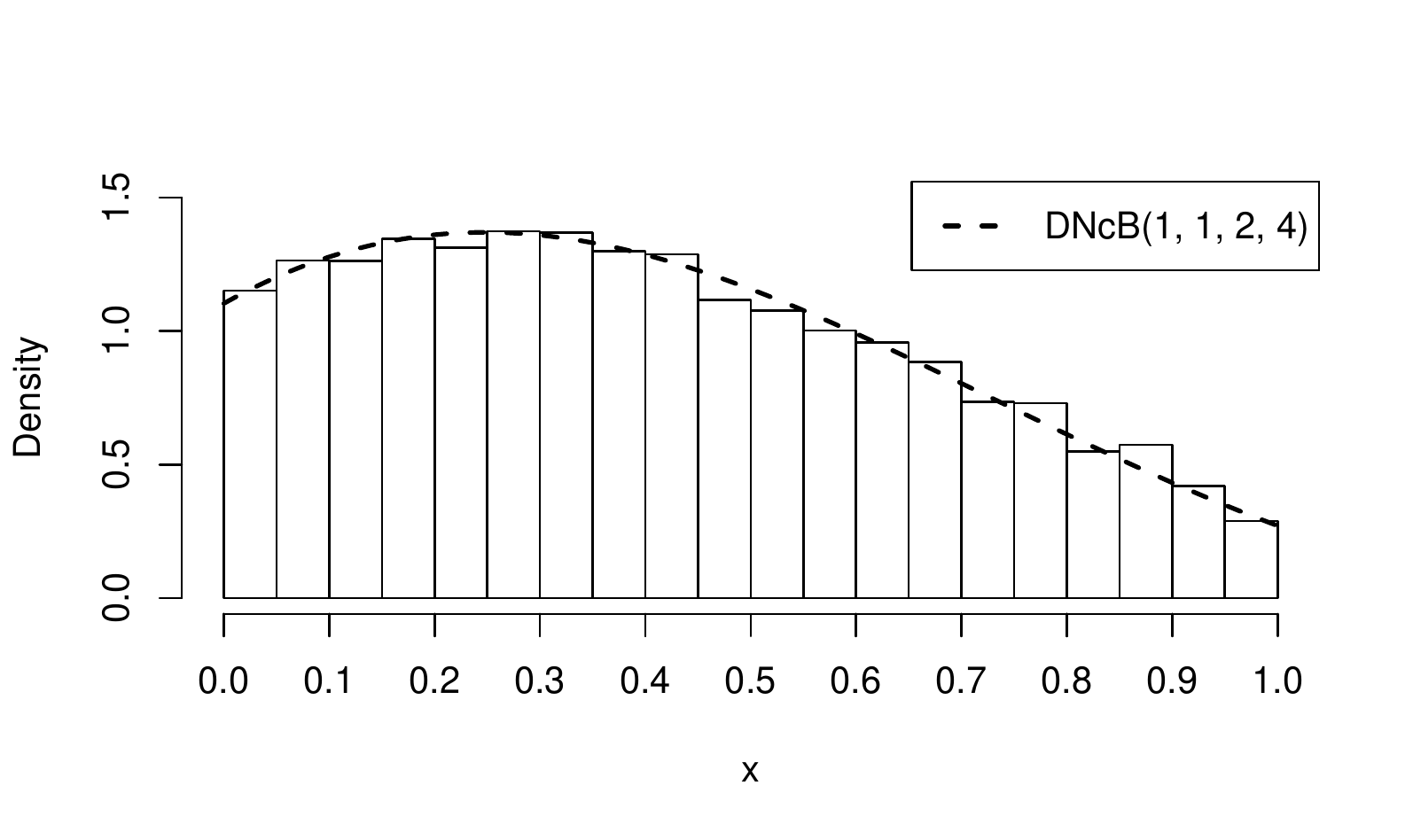}}
 \subfigure
   {\includegraphics[width=6.6cm]{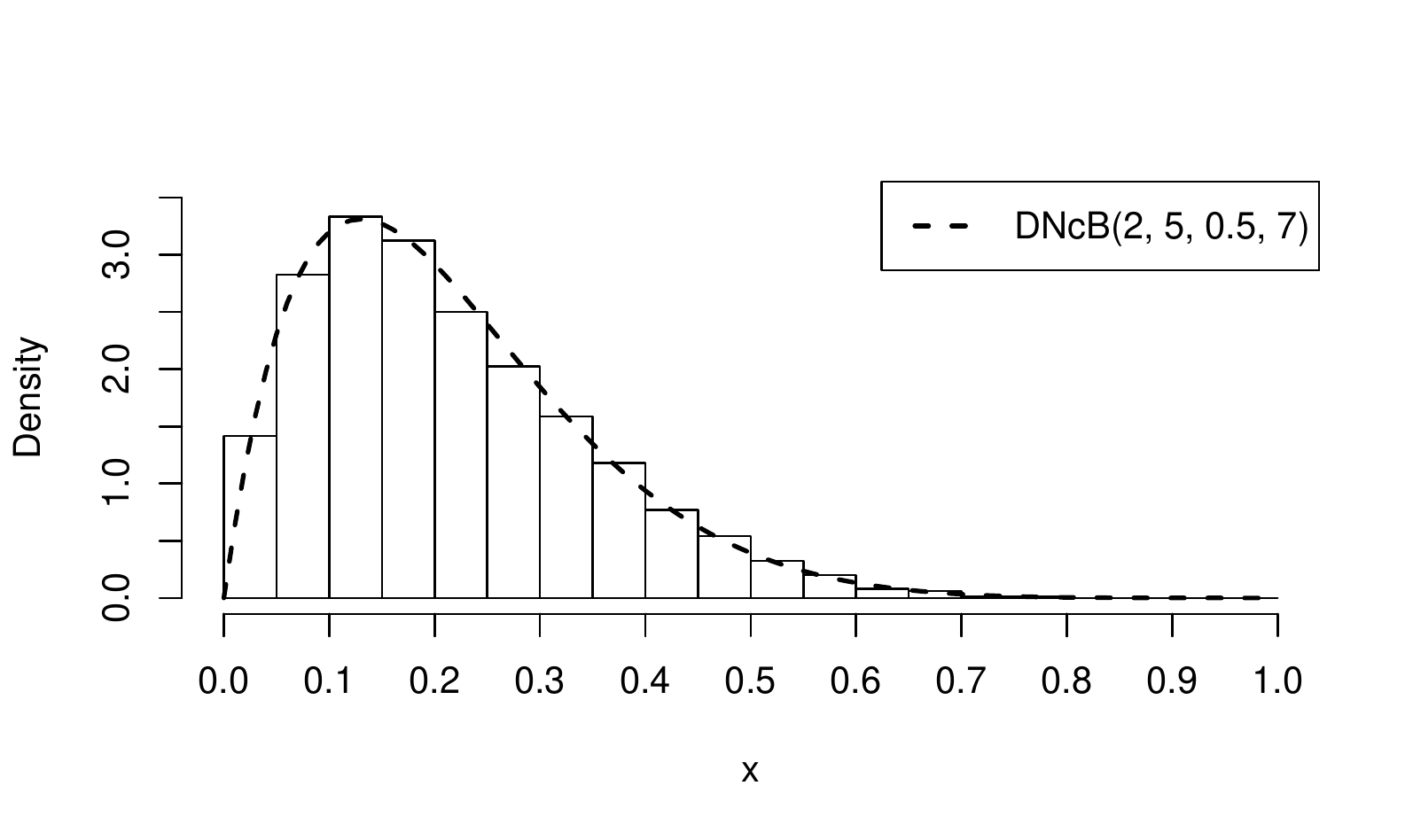}}
  \caption{Histograms of 10000 random draws simulated from the $\mbox{DNcB}(\alpha_1,\alpha_2,\lambda_1,\lambda_2)$ model by means of the conditional density of the latter given $M^+$ in Equation~(\ref{eq:dncb.distr.cond.m}) for the following values of the parameter vector: $\left(0.5,0.5,4,4\right)$ in the top left-hand panel, $\left(0.5,0.5,4,7\right)$ in the top right-hand panel, $\left(1,1,2,4\right)$ in the bottom left-hand panel and $\left(2,5,0.5,7\right)$ in the bottom right-hand panel; the plot of the true densities is superimposed in black.}
 \label{fig:SIMUL}
\end{figure}
That said, numerical validations of the derived formulas are produced by employing the following procedure. A sample of $n=30$ series of 10000 random draws is generated from each of the two models of interest for selected values of their respective parameters. The $r$-th descriptive moment $(r=1,2,3,4)$ is thus computed for each series; then, the sample mean and standard deviation of this latter quantity are assessed for every $r$. Hence, the null hypothesis that the true mean is equal to the value of the $r$-th moment about zero of the considered model is checked by using the two-tailed $Z$ test for large samples. The results obtained are listed in Table~\ref{tab:moment.test.ncchisq} for the $\chi'^{\,2}$ distribution and in Table~\ref{tab:moment.test.dncb} for the DNcB one. All the conclusions are clearly in favour of the non-rejection of the null hypothesis at study.

\begin{table}[ht]
\caption{Sample means ($\bar{x}$) and standard deviations ($s$) of the $r$-th descriptive moments ($r=1,2,3,4$) of $n=30$ series of 10000 random draws from the $\chi'^{\,2}_g \left(\lambda \right)$ distribution for selected values of its parameters and $p \, $-values of the two-tailed $Z$ test to check the null hypothesis that the true mean is equal to the $r$-th moment about zero of the model at study.}

\vspace{0.25cm}
\centering
\begin{tabular}{cc||c|c|c|c|c}
$g$ & $\lambda$ & $r$ & $r$-th moment & $\bar{x}$ & $s$ & $p-$value\\ \hline \hline
\multirow{4}{0.5cm}{ 2}  &   \multirow{4}{0.5cm}{ 4} & 1  &  6   &     6.00339      &   0.04723  & 0.69422\\
\cline{3-7} & &     \multicolumn{1}{c| }{2} & \multicolumn{1}{c| }{56}      & \multicolumn{1}{c| }{56.04660}     & \multicolumn{1}{c| }{0.90051} & \multicolumn{1}{c}{0.77684}\\ 
\cline{3-7} & &     \multicolumn{1}{c| }{3} & \multicolumn{1}{c| }{688}     & \multicolumn{1}{c| }{688.63642}   & \multicolumn{1}{c| }{19.44960} & \multicolumn{1}{c}{0.85776} \\ 
\cline{3-7} & &     \multicolumn{1}{c| }{4} & \multicolumn{1}{c| }{10368}   & \multicolumn{1}{c| }{10370.10193} & \multicolumn{1}{c| }{481.57720} & \multicolumn{1}{c}{0.98093}\\ \hline \hline
\multirow{4}{0.5cm}{4.5}  & \multirow{4}{0.5cm}{ 2} & 1  & 6.5  &     6.49742     &    0.03293  & 0.66783\\
\cline{3-7} & &     \multicolumn{1}{c| }{2} & \multicolumn{1}{c| }{59.25}      & \multicolumn{1}{c| }{59.23174}   & \multicolumn{1}{c| }{0.64579} & \multicolumn{1}{c }{0.87692}\\ 
\cline{3-7} & &     \multicolumn{1}{c| }{3} & \multicolumn{1}{c| }{690.125}    & \multicolumn{1}{c| }{689.83859}  & \multicolumn{1}{c| }{13.54150} & \multicolumn{1}{c}{0.90777} \\ 
\cline{3-7} & &     \multicolumn{1}{c| }{4} & \multicolumn{1}{c| }{9745.56250} & \multicolumn{1}{c| }{9730.61264} & \multicolumn{1}{c| }{323.74687} & \multicolumn{1}{c}{0.80033}\\ \hline \hline
\multirow{4}{0.5cm}{ 3}  & \multirow{4}{0.5cm}{1.5} & 1  & 4.5  &     4.49802     &    0.03401  & 0.74982\\
\cline{3-7} & &     \multicolumn{1}{c| }{2} & \multicolumn{1}{c| }{32.25}      & \multicolumn{1}{c| }{32.20308}   & \multicolumn{1}{c| }{0.53281} & \multicolumn{1}{c }{0.62957}\\ 
\cline{3-7} & &     \multicolumn{1}{c| }{3} & \multicolumn{1}{c| }{313.125}    & \multicolumn{1}{c| }{312.66438}  & \multicolumn{1}{c| }{10.25106} & \multicolumn{1}{c}{0.80559} \\ 
\cline{3-7} & &     \multicolumn{1}{c| }{4} & \multicolumn{1}{c| }{3812.0625}  & \multicolumn{1}{c| }{3812.18461} & \multicolumn{1}{c| }{229.93738} & \multicolumn{1}{c}{0.99768}\\ \hline \hline
\multirow{4}{0.5cm}{ 6}  & \multirow{4}{0.5cm}{3.5} & 1  & 9.5  &     9.49716     &    0.05716  & 0.78552\\
\cline{3-7} & &  \multicolumn{1}{c| }{2} & \multicolumn{1}{c| }{116.25}     & \multicolumn{1}{c| }{116.17581}  & \multicolumn{1}{c| }{1.44904} & \multicolumn{1}{c}{0.77915}\\ 
\cline{3-7} & &  \multicolumn{1}{c| }{3} & \multicolumn{1}{c| }{1730.375}   & \multicolumn{1}{c| }{1729.08064} & \multicolumn{1}{c| }{36.31869} & \multicolumn{1}{c}{0.84523} \\ 
\cline{3-7} & &  \multicolumn{1}{c| }{4} & \multicolumn{1}{c| }{30228.06250} & \multicolumn{1}{c| }{30198.26443} & \multicolumn{1}{c| }{988.70828} & \multicolumn{1}{c}{0.86889}
\end{tabular}
\small
\label{tab:moment.test.ncchisq}
\end{table}

\begin{table}[ht]
\caption{Sample means ($\bar{x}$) and standard deviations ($s$) of the $r$-th descriptive moments ($r=1,2,3,4$) of $n=30$ series of 10000 random draws from the $\mbox{\normalfont{DNcB}}\left(\alpha_1,\alpha_2,\lambda_1,\lambda_2\right)$ distribution for selected values of its parameters and $p \, $-values of the two-tailed $Z$ test to check the null hypothesis that the true mean is equal to the $r$-th moment about zero of the model at study.}

\vspace{0.25cm}
\centering
\begin{tabular}{cccc||c|c|c|c|c}
$\alpha_1$ & $\alpha_2$ & $\lambda_1$ & $\lambda_2$ & $r$ & $r$-th moment & $\bar{x}$ & $s$ & $p-$value\\ \hline \hline
\multirow{4}{0.5cm}{0.5} & \multirow{4}{0.5cm}{0.5} & \multirow{4}{0.5cm}{ 4} & \multirow{4}{0.5cm}{ 4} & 1 & 0.5 & 0.50011 & 0.00299 & 0.84031\\
\cline{5-9} & & & & \multicolumn{1}{c| }{2} & \multicolumn{1}{c| }{0.33013} & \multicolumn{1}{c| }{0.33011} & \multicolumn{1}{c| }{0.00305} & \multicolumn{1}{c}{0.97382}\\ 
\cline{5-9} & & & & \multicolumn{1}{c| }{3} & \multicolumn{1}{c| }{0.24519} & \multicolumn{1}{c| }{0.24509} & \multicolumn{1}{c| }{0.00287} & \multicolumn{1}{c}{0.84505} \\ 
\cline{5-9} & & & & \multicolumn{1}{c| }{4} & \multicolumn{1}{c| }{0.19516} & \multicolumn{1}{c| }{0.19500} & \multicolumn{1}{c| }{0.00268} & \multicolumn{1}{c}{0.75117}\\ \hline \hline
\multirow{4}{0.5cm}{0.5} & \multirow{4}{0.5cm}{0.5} & \multirow{4}{0.5cm}{ 4} & \multirow{4}{0.5cm}{ 7} & 1 & 0.38833 & 0.38830 & 0.00224 & 0.94458\\
\cline{5-9} & & & & \multicolumn{1}{c| }{2} & \multicolumn{1}{c| }{0.21345} & \multicolumn{1}{c| }{0.21334} & \multicolumn{1}{c| }{0.00196} & \multicolumn{1}{c }{0.76246}\\ 
\cline{5-9} & & & & \multicolumn{1}{c| }{3} & \multicolumn{1}{c| }{0.13759} & \multicolumn{1}{c| }{0.13751} & \multicolumn{1}{c| }{0.00176} & \multicolumn{1}{c}{0.79195} \\ 
\cline{5-9} & & & & \multicolumn{1}{c| }{4} & \multicolumn{1}{c| }{0.09788} & \multicolumn{1}{c| }{0.09784} & \multicolumn{1}{c| }{0.00161} & \multicolumn{1}{c}{0.87988}\\ \hline \hline
\multirow{4}{0.5cm}{ 1} & \multirow{4}{0.5cm}{ 1} & \multirow{4}{0.5cm}{ 2} & \multirow{4}{0.5cm}{ 4}  & 1  & 0.40925 & 0.40921 & 0.00220 & 0.91796\\
\cline{5-9} & & & & \multicolumn{1}{c| }{2} & \multicolumn{1}{c| }{0.23211} & \multicolumn{1}{c| }{0.23202} & \multicolumn{1}{c| }{0.00203} & \multicolumn{1}{c}{0.80964}\\ 
\cline{5-9} & & & & \multicolumn{1}{c| }{3} & \multicolumn{1}{c| }{0.15356} & \multicolumn{1}{c| }{0.15345} & \multicolumn{1}{c| }{0.00176} & \multicolumn{1}{c}{0.72559} \\ 
\cline{5-9} & & & & \multicolumn{1}{c| }{4} & \multicolumn{1}{c| }{0.11134} & \multicolumn{1}{c| }{0.11123} & \multicolumn{1}{c| }{0.00155} & \multicolumn{1}{c}{0.68663}\\ \hline \hline
\multirow{4}{0.5cm}{ 2} & \multirow{4}{0.5cm}{ 5} & \multirow{4}{0.5cm}{0.5} & \multirow{4}{0.5cm}{ 7} & 1  & 0.21392 & 0.21388 & 0.00140 & 0.86706\\
\cline{5-9} & & & & \multicolumn{1}{c| }{2} & \multicolumn{1}{c| }{0.06298} & \multicolumn{1}{c| }{0.06296} & \multicolumn{1}{c| }{0.00074} & \multicolumn{1}{c}{0.87554}\\ 
\cline{5-9} & & & & \multicolumn{1}{c| }{3} & \multicolumn{1}{c| }{0.02281} & \multicolumn{1}{c| }{0.02279} & \multicolumn{1}{c| }{0.00039} & \multicolumn{1}{c}{0.77032} \\ 
\cline{5-9} & & & & \multicolumn{1}{c| }{4} & \multicolumn{1}{c| }{0.00957} & \multicolumn{1}{c| }{0.00956} & \multicolumn{1}{c| }{0.00022} & \multicolumn{1}{c}{0.71113}
\end{tabular}
\small
\label{tab:moment.test.dncb}
\end{table}

The new formula for the moments of the Non-central Chi-Squared distribution in Equation~(\ref{eq:mom.ncchisq}) does not offer any particular advantage of a computational or numerical nature over the existing one in Equation~(\ref{eq:mom.literat.ncchisq}) but only brings an element of theoretical elegance into the study of the above mentioned model. On the contrary, the new formula for the moments of the DNcB distribution in Equation~(\ref{eq:momr.beta.dnc}) allows to overcome the disadvantages due to the infinite-series structure of the existing formula in Equation~(\ref{eq:momr.beta.dnc.oneser}). As a matter of fact, the former is computationally less demanding than the latter thanks to its finite-sum form. Specifically, the derived moment formula requires a less implementation effort and calls for a less considerable execution time to produce the desired results than the existing one. In this regard, the computational performances of the two formulas are compared in the following way. The overall amount of machine-time spent to compute the first four moments of the DNcB distribution by using both the derived and the existing formulas is measured $n=30$ times for each value of the parameter vector considered in Table~\ref{tab:moment.test.dncb}. Then, the sample mean and standard deviation of this latter quantity are computed in each case above. Hence, the null hypothesis that the true mean of the execution time of the new formula is not inferior to the one of the existing formula is checked by using the one-tailed $Z$ test for large samples. The results achieved are listed in Table~\ref{tab:moment.time.dncb}, from which we conclude that the new formula is beyond doubt to be preferred to the existing one for its major computational efficiency. More precisely, the findings obtained suggest that the execution time of the new formula is approximately of the order of 5 times faster than the one of the existing formula.

\begin{table}[ht]
\caption{Means ($\bar x$) and standard deviations ($s$) of the overall machine-time (in seconds) to compute the first four moments of the DNcB distribution by means of the derived formula (``Sum'') and the existing formula (``Series'') for selected values of the parameter vector and $p \, $-values of the one-tailed $Z$ test to check the null hypothesis $H_0: \, \mu_{Sum}-\mu_{Series} \, \geq\, 0$ ($n=30$).}

\vspace{0.25cm}
\centering
\begin{tabular}{cccc||c|c|c|c}
\multirow{2}{0.5cm}{$\alpha_1$}  & \multirow{2}{0.5cm}{$\alpha_2$} & \multirow{2}{0.5cm}{$\lambda_1$} & \multirow{2}{0.5cm}{$\lambda_2$} & \multicolumn{4}{c}{Time ($''$)} \\ 
\cline{5-8} &     &   &   &  Formula                         & $\bar{x}$                     & $s$                           & $p$-value \\ \hline \hline
\multirow{2}{0.5cm}{0.5} & \multirow{2}{0.5cm}{0.5} & \multirow{2}{0.25cm}{4} & \multirow{2}{0.25cm}{4} & Sum & 0.00567 & 0.00817 & \multirow{2}{1.3cm}{$<.0001$} \\
\cline{5-7} &     &     &   & \multicolumn{1}{c| }{Series} & \multicolumn{1}{c| }{0.03200} & \multicolumn{1}{c| }{0.01297} & \multicolumn{1}{c}{} \\ \hline \hline
\multirow{2}{0.5cm}{0.5} & \multirow{2}{0.5cm}{0.5} & \multirow{2}{0.25cm}{4} & \multirow{2}{0.25cm}{7} & Sum & 0.00533 & 0.00776 & \multirow{2}{1.3cm}{$<.0001$} \\
\cline{5-7} &     &     &   & \multicolumn{1}{c| }{Series} & \multicolumn{1}{c| }{0.03500} & \multicolumn{1}{c| }{0.01280} & \multicolumn{1}{c}{} \\ \hline \hline
\multirow{2}{0.25cm}{1} & \multirow{2}{0.25cm}{1} & \multirow{2}{0.25cm}{2} & \multirow{2}{0.25cm}{4} & Sum   & 0.00467 & 0.00730 & \multirow{2}{1.3cm}{$<.0001$}   \\
\cline{5-7} &     &     &   & \multicolumn{1}{c| }{Series} & \multicolumn{1}{c| }{0.02500} & \multicolumn{1}{c| }{0.01383} & \multicolumn{1}{c}{} \\ 
\hline \hline
\multirow{2}{0.25cm}{2} & \multirow{2}{0.25cm}{5} & \multirow{2}{0.5cm}{0.5} & \multirow{2}{0.25cm}{7} & Sum  & 0.00533 & 0.00819 & \multirow{2}{1.3cm}{$<.0001$}   \\
\cline{5-7} &     &     &   & \multicolumn{1}{c| }{Series} & \multicolumn{1}{c| }{0.01767} & \multicolumn{1}{c| }{0.00728} & \multicolumn{1}{c}{} 
\end{tabular}
\small
\label{tab:moment.time.dncb}
\end{table}

All the simulations and the procedures of interest have been carried out by means of routines implemented in the \texttt{R} environment. 

\section{Concluding remarks}
\label{sec:concl}

In the present paper the problem of computing the $r$-th moment about zero of the Non-central Chi-Squared distribution was faced by exploiting a conditional approach based on the mixture representation of such a distribution together with a new expansion of the ascending factorial of a binomial. A similar approach enabled the derivation of a new formula for the $r$-th moment about zero of the Doubly Non-central Beta distribution, i.e. a generalization of the Beta model for the definition of which the Non-central Chi-Squared distribution represents the main ingredient. While the new formula for the moments of the Non-central Chi-Squared distribution only brings an element of theoretical elegance into the study of this model, the new formula for the moments of the Doubly Non-central Beta distribution, thanks to its finite-sum form, guarantees a greater computational efficiency than the existing one, the latter having an infinite-series structure. We hope that the approaches discussed and the results obtained may attract wider applications in statistics.



\begin{thebibliography}{00}

\bibitem[Abramowitz and Stegun (1964)]{AbrSte64} Abramowitz, M. and Stegun, I. A. 1964. \textit{Handbook of Mathematical Functions with Formulas, Graphs and Mathematical Tables}. Washington D.C.: U.S. Government Printing Office.
%
\bibitem[Gould (1972)]{Gou72} Gould, H. W. 1972. \textit{Combinatorial Identities}. Morgantown (WV): Morgantown Printing and Binding Company.
%
\bibitem[Johnson \textit{et al.} (1994)]{JohKotBal94} Johnson, N. L., Kotz, S. and Balakrishnan, N. 1994. Vol. 1 of \textit{Continuous Univariate Distributions}, 2nd edn. New York: John Wiley \& Sons.
%
\bibitem[Johnson \textit{et al.} (1995)]{JohKotBal95} Johnson, N. L., Kotz, S. and Balakrishnan, N. 1995. Vol. 2 of \textit{Continuous Univariate Distributions}, 2nd edn. New York: John Wiley \& Sons.
%
\bibitem[Johnson \textit{et al.} (2005)]{JohKemKot05} Johnson, N. L., Kemp, A. W. and Kotz, S. 2005. \textit{Univariate Discrete Distributions}, 3rd edn. Hoboken (NJ): John Wiley \& Sons.
%
\bibitem[Kimball and Scheibner (1998)]{KimSch98} Kimball, C. V. and Scheibner, D. J. 1998. Error Bars for Sonic Slowness Measurements. Geophysics, 63:(2) 345-353.
%
\bibitem[Nadarajah and Gupta (2004)]{NadGup04} Nadarajah, S. and Gupta, A. K. 2004. \textit{Handbook of Beta Distribution and Its Applications}. New York: Marcel Dekker, Inc.

\bibitem[Ongaro and Orsi (2015)]{OngOrs15} Ongaro, A. and Orsi, C. 2015. Some Results on Non-Central Beta Distributions. \textit{Statistica} 75:(1) 85-100.
%
\bibitem[Siegel (1979)]{Sie79} Siegel, A. F. 1979. The Noncentral Chi-Squared Distribution with Zero Degrees of Freedom and Testing for Uniformity. Biometrika, 66:(2) 381-386.
%
\bibitem[Srivastava and Karlsson (1985)]{SriKar85} Srivastava, H. M. and Karlsson, W. 1985. \textit{Multiple Gaussian Hypergeometric Series}. Chichester (UK): Ellis Horwood. 
%
\bibitem[Stamm \textit{et al.} (2015)]{Sta15} Stamm, A., Singh, J., Afacan, O. and Warfield, S. K. 2015. Analytic Quantification of Bias and Variance of Coil Sensitivity Profile Estimators for Improved Image Reconstruction in  \uppercase{mri}. In \textit{Navab, N., Hornegger, J., Wells, W. and Frangi A. (eds.) Medical Image Computing and Computer-Assisted Intervention - MICCAI 2015. Lecture Notes in Computer Science} 9350: 684-691. Cham (CH): Springer.
%
\end{thebibliography}
\end{document}